\documentclass[12pt]{article}
\usepackage[utf8]{inputenc}
\usepackage{amsmath}
\usepackage{amsfonts}
\usepackage{amssymb}
\usepackage{enumerate}
\usepackage{mathtools}
\usepackage{soul}
\usepackage{enumitem}
\usepackage{adjustbox}
\usepackage{indentfirst}
\usepackage{xcolor}
\usepackage{esint}
 \usepackage{authblk}
\usepackage{eucal}

\usepackage[normalem]{ulem}

\renewcommand{\sf}[1]{\mathsf{#1}}

\renewcommand{\bar}[1]{\overline{#1}}
\renewcommand{\tilde}{\widetilde}
\renewcommand{\hat}{\widehat}
\newcommand{\pt}{\partial}

\DeclareMathOperator{\vol}{vol}
\DeclareMathOperator{\Vol}{Vol}

\DeclareMathOperator{\Ric}{\sf{Ric}}

\newcommand{\dd}[1]{\ensuremath{\operatorname{d}\!{#1}}}

\DeclareMathOperator{\supp}{supp}
\DeclareMathOperator{\Tan}{Tan}
\DeclareMathOperator{\id}{id}

\mathchardef\mhyphen="2D

\usepackage{amsthm}
\usepackage{aliascnt}
\usepackage{hyperref}
\usepackage{cleveref}

\theoremstyle{plain}
\newtheorem{theorem}{Theorem}[section]

\theoremstyle{definition}

 \newaliascnt{lemma}{theorem} 
 \newaliascnt{proposition}{theorem} 
 \newaliascnt{corollary}{theorem} 
 \newaliascnt{remark}{theorem} 
 \newaliascnt{example}{theorem} 

 \newtheorem{definition}[theorem]{Definition}
 \newtheorem{proposition}[proposition]{Proposition}
 \newtheorem{corollary}[corollary]{Corollary}
 \newtheorem{lemma}[lemma]{Lemma}
 \newtheorem{example}[example]{Example}
 
 \theoremstyle{remark}
 \newtheorem{remark}[remark]{Remark}
 
 \aliascntresetthe{corollary}
 \aliascntresetthe{proposition}
 \aliascntresetthe{lemma}
 \aliascntresetthe{remark}
 \aliascntresetthe{example}

 \usepackage{titlesec}
 \titleformat{\section}{\large\bfseries}{\thesection.}{0.5em}{}

\titleformat{\subsection}
  {\normalfont\normalsize\bfseries}  
  {\thesubsection.}                   
  {0.5em}                             
  {}                                  

    \usepackage[
      paper=a4paper,
      top=1.2in,
      bottom=1.2in,
      left=1.2in,
      right=1.2in
    ]{geometry}

\usepackage{setspace}
\title{\large\bfseries
DIMENSION BOUND OF SINGULAR SET OF ONE-PHASE FREE BOUNDARY PROBLEMS IN SPACES WITH TWO-SIDED RICCI BOUND}

\author{Kai-Hsiang Wang}
\affil{Faculty of Mathematics, Technion - Israel Institute of Technology\\
khwang2025@campus.technion.ac.il}
\date{\today}

\begin{document}
\maketitle
\vspace{-1cm}
\begin{abstract}
For one-phase free boundary problems in $\mathbb{R}^n$, the Hausdorff dimension of the singular set of the free boundary can be bounded by $n-k^*$, where $k^*$ is the minimal dimension such that there exist nonlinear $1$-homogeneous global minimizers to the one-phase problem on $\mathbb{R}^{k^*}$.
From \cite{CJK04, DJ09, JS15}, it is known that $k^*\in \{5,6,7\}$, although the exact value is still an open problem.
\par
In \cite{CZZ22}, the authors extended this study to singular spaces by considering the one-phase problem in noncollapsed $\mathsf{RCD}(K,N)$ spaces, which, in a synthetic sense, have Riemannian Ricci curvature bounded from below by $K$ and the dimension bounded from above by $N$.
With singularities arising possibly from the ambient space and from the free boundary structure, they showed that the singular set of the free boundary has the sharp Hausdorff dimension bound $N-3$.
\par
In this article, we show that for noncollapsed limits of $n$-dimensional manifolds with stronger two-sided Ricci curvature bounds, the Hausdorff dimension bound can be improved to $n-5$, which is sharp in this context.

\end{abstract}
\tableofcontents

\setlength{\parskip}{0.5em}

\section{Introduction and Main Results}

Let $(X,d,\mu)$ be a suitable metric measure space and $\Omega\subset X$ a bounded open domain.
The one-phase Bernoulli free boundary problem is concerned with minimizers of the Alt--Caffarelli functional:
\begin{equation}\label{1-phase_Ber}
    E(v,\Omega):= \int_\Omega |\nabla v|^2 + Q^2 1_{v>0} \dd \mu,
\end{equation}
among $v\in W^{1,2}(\Omega)$ with $v-v_0\in W^{1,2}_0(\Omega)$ for some given nonnegative $v_0\in W^{1,2}(\Omega)$ and strictly positive $Q \in C^{\alpha}(\Omega)$.
By considering the positive part, we may assume that a minimizer $u$ is nonnegative, and hence the name ``one-phase".
In the case when $\Omega$ is not bounded, we require that $u\in W^{1,2}_{\text{loc}}(\Omega)$ minimizes $E(\cdot, \Sigma)$ over any bounded subdomains $\Sigma\subset \Omega$ with its boundary data.
In particular, when $\Omega=X$, we say that $u$ is a global minimizer.
Since the boundary of the positive part, $\pt \{u>0\}\cap \Omega$, is determined by $u$ itself, it is called the ``free boundary" in this context.
Formally, $u$ solves the following boundary value problem with the free boundary:
\begin{equation}\label{eq:PDE}
    \begin{cases}
        \Delta u = 0& \text{on $\{u>0\}\cap \Omega$};\\
        |\nabla u|=Q& \text{on $\pt \{u>0\}\cap \Omega$};\\
        u \geq 0 &\text{on $\Omega$};\\
        u=v_0 & \text{on $\pt \Omega$}.
    \end{cases}
\end{equation}

Fundamental questions involve the regularity of minimizers and the associated free boundary, which can be further divided into a \textit{regular part} and a \textit{singular part}.
Hence, one studies how smooth the regular part is, and how small the singular part is.
\par
When $X$ is the Euclidean space with the canonical metric measure structure, this free boundary problem has been intensively studied in the literature. 
The pioneer work \cite{AC81} established the existence and Lipschitz continuity of minimizers $u$, and showed that $u$ essentially solves the boundary value problem \eqref{eq:PDE}.
For the free boundary, they proved that it is a $C^{1,\alpha}$-manifold outside a relatively closed set $\mathcal{S}(u)$ with zero Hausdorff $(n-1)$-measure.
$\mathcal{S}(u)$ is the singular set in this context.
\par
More systematic study of the singular set in this setting begins with \cite{Wei99}, who introduced the Weiss monotonicity formula (see Theorem \ref{thm:orginal_weiss}) and confirmed the conjecture in \cite{AC81} that blowups of minimizers are $1$-homogeneous.
Together with Federer's dimension reduction argument, the Hausdorff dimension bound on the singular set was improved to $\dim_H (\mathcal{S}(u))\leq n - k^*$, where:
\begin{definition}\label{def:crit_dim}
    The \emph{critical dimension} $k^*\in \mathbb{N}$ is the smallest dimension such that $\mathbb{R}^{k^*}$ admits a nonlinear $1$-homogeneous global minimizer of the one-phase problem \eqref{1-phase_Ber} with constant $Q$.
\end{definition}
\noindent That is to say, for $n<k^*$, any nontrivial $1$-homogeneous global minimizer  $u$ on $\mathbb{R}^n$ with constant $Q$ is exactly
\begin{equation}\label{eq:linear_sol}
    h(x)=Q(x\cdot v)^+
\end{equation}
for some unit vector $v$, and we call such solutions \emph{linear solutions}. 
It can easily be shown that $k^*$ is independent of $Q$.
With the works \cite{CJK04, JS15, DJ09}, we know $k^* \in \{5,6,7\}$, although the exact value is still unknown. 
\par
In \cite{CZZ22}, the study of this free boundary problem is extended to non-smooth spaces. The ambient spaces are now the so-called $\mathsf{RCD}(K,N)$ spaces, which are metric measure spaces that have Riemannian Ricci curvature bounded from below by $K$, and the dimension bounded from above by $N$ in a synthetic sense.
They also assume that the space is noncollapsed in volume, which amounts to saying that the reference measure is the $N$-dimensional Hausdorff measure $\mathcal{H}^N$. 
Among their results, they recovered most of the regularity results in \cite{AC81} for minimizers in this non-smooth setting.
As for the free boundary, notice a subtlety here that the singularities can now come from either the free boundary structure itself or from the background space.
To facilitate our discussion, we introduce some notations.
\par
\textbf{Notations}\label{intro:notations}: Given a noncollapsed $\mathsf{RCD}(K,N)$ space $(X,d,\mathcal{H}^N)$, let $\mathcal{R}(X)$ denote the regular set of $X$; its complement is the singular set, denoted by $\mathcal{S}(X)$ (see Definition \ref{def:ric_lim_reg_sing}).
The regular and singular sets of the free boundary $\pt \{u>0\}\cap \Omega$ are denoted respectively by $\mathcal{R}(u)$ and $\mathcal{S}(u)$ (see Definition \ref{def:RCD_FBP_reg_sing}). 
We also denote by $\mathcal{S}^k(X)$ and $\mathcal{S}^k(u)$ respectively the $k$-singular stratum of $X$ and $\pt \{u>0\}\cap \Omega$ (see Definition \ref{def:ric_lim_strata} and \ref{def:strata_both_asp}).
\par
In this situation of possible mixed singularities in the free boundary, \cite[Corollary 1.11]{CZZ22} (also see Theorem \ref{thm:RCD_free_bdy_reg}) proved that 
\begin{equation}\label{eq:intro_RCD_FBP_sing_sharp_bd}
    \dim_H \mathcal{S}(u)\leq N-3,
\end{equation}
and this dimension bound is sharp in this non-smooth setting.
The same dimension bound was proved in \cite{MS25} for the singular set of perimeter minimizing sets in noncollapsed $\mathsf{RCD}$ spaces.
These bounds should be compared with $ \dim_H \mathcal{S}(X)\leq N-2$ from \cite{DG18} about the singular set of the background $\mathsf{RCD}$ space $X$ itself (assuming $X$ has no boundary).
In addition, \cite{CZZ22} also showed a Reifenberg-type result that outside a subset of $\mathcal{S}(u)$, the free boundary has a $C^{\alpha}$ manifold structure
A similar result was proved in \cite{MS25} for perimeter minimizing sets, and both are parallel to \cite{KM21} about the $C^\alpha$ manifold structure outside a subset of $\mathcal{S}(X)$.

\par
Motivated by \cite{CF24} about perimeter minimizing sets in noncollapsed limit spaces with stronger two-sided bounds on the Ricci curvature (recall that $\mathsf{RCD}$ spaces only have Ricci lower bounds), we consider the one-phase problem in such spaces. 
By definition, these are pointed metric spaces $(X, d, x)$ arising as pointed Gromov--Hausdorff limits of sequences of pointed complete Riemannian manifolds $(M^n_j , d_j , x_j )$ of some fixed dimension $n$, where
$d_j$ denotes the Riemannian distance, satisfying
\begin{equation}\label{intro:2sided_Ric_lim}
    |\Ric_{M^n_j}|\leq n-1 \quad \text{and}\quad \Vol(B_1(x_j))\geq v>0
\end{equation}
for some constant $v$.
Naturally, $X$ is equipped with the $n$-dimensional Hausdorff measure $\mathcal{H}^n$.
We refer to \cite{And89, And90, CC96,CC97,CN15} about studies of such limit spaces, and point out that they fall within the framework of noncollapsed $\mathsf{RCD}$ spaces (see \cite{Gig15}).
In particular, all fundamental results in \cite{CZZ22} are valid in our case.
We also mention that with the stronger curvature assumption in \eqref{intro:2sided_Ric_lim}, it is proved in \cite{And90} and \cite{CN15} that $S(X)$ is a closed set and has a better bound $\dim_H \mathcal{S}(X)\leq n-4$.

\par

\subsection{Main Results}
We show that under the stronger assumption of two-sided bounds on the Ricci curvature, the dimension bound of the singular set of the one-phase problem in the non-smooth setting can be improved: 
\begin{theorem}\label{thm:dim_bound}
    Let $X$ be an $n$-dimensional noncollapsed limit space with two-sided bounds on the Ricci curvature, as defined by \eqref{intro:2sided_Ric_lim}.
    Let $u$ be a minimizer of the one-phase problem \eqref{1-phase_Ber} on some bounded domain $\Omega\subset X$.
    Then we have $\mathcal{S}(u)=\mathcal{S}^{n-5}(u)$, with the \hyperref[intro:notations]{\textbf{Notations}} above.
    In particular, we have
    \begin{equation}
        \dim_{H}\mathcal{S}(u)\leq n-5.
    \end{equation}
Furthermore, this dimension bound is sharp.
\end{theorem}
\begin{remark}[The behavior of $\pt \{u>0\}$ in $\mathcal{R}(X)$]\label{rmk:FB_reg_in_R(X)}
    By \cite{And90,CC96}, $\mathcal{R}(X)$ is open and has a $C^{1,\beta}$ manifold structure, for any $\beta\in (0,1)$.
    Hence the arguments from \cite{AC81, Wei99} in the Euclidean case can be adapted to show that $\pt \{u>0\}\cap \mathcal{R}(X)$ is a $C^{1,\gamma}$ hypersurface in $\mathcal{R}(X)$, for some $\gamma$, outside of a closed set of Hausdorff dimension at most $n-k^*$, where $k^*$ is the critical dimension.
    See Proposition \ref{prop:FB_reg_in_R(X)}, and compare this to \cite[Remark 1.1]{CF24}.
\end{remark}
The sharpness of Theorem \ref{thm:dim_bound} is actually due to potential singularities coming from the ambient space (see Example \ref{ex:dim_bound_sharp}), which is the same reason behind the sharpness of \cite[Corollary 1.11]{CZZ22} and \cite[Theorem 1]{CF24}.
Therefore, it cannot be improved despite possible improvements in the critical dimension.
\par
The proof of Theorem \ref{thm:dim_bound} is similar to that of Theorem 1 in \cite{CF24}, where they proved the same dimension bound for the singular set of perimeter minimizing sets in such limit spaces.
It is done by analyzing the blowups of $u$, which are global minimizers on the corresponding tangent cone of $X$ at the blowup point, together with the following rigidity statement:
\begin{theorem}\label{thm:rigidity}
Let $\Gamma$ be a discrete group of isometries on the unit sphere $S^3$ acting freely.
    Let $C(S^3/\Gamma)$ be the metric cone over $S^3/\Gamma$, and let $p$ be its tip.
    If $u$ is a global minimizer of the one-phase problem \eqref{1-phase_Ber} on $C(S^3/\Gamma)$ with constant $Q$ such that $p \in \pt \{u>0\}$, then $\Gamma$ is trivial, and $u$ is a linear solution on $C(S^3)\cong \mathbb{R}^4$ as in \eqref{eq:linear_sol}.
\end{theorem}
To prove Theorem \ref{thm:rigidity}, we first modify the symmetry tools developed in \cite{CF24} to transform the problem into a Euclidean symmetric setting.
After that, we modify the arguments in \cite{CJK04,JS15} to make them applicable to the symmetric setting.
The connection between this theorem and Theorem \ref{thm:dim_bound} is that the tangent cones of $X$ at any $x\in \mathcal{S}(u)\backslash\mathcal{S}^{n-5}(u)$ are exactly those metric cones in this theorem with some extra Euclidean factors, thanks to \cite{And90, CN15}, and the classical classification of manifolds with constant sectional curvature (see Theorem \ref{thm:twosided_ric_n-4_rig}).
\par
Although the dimension bound in the previous Theorem \ref{thm:dim_bound} is sharp, one may wonder if Theorem \ref{thm:rigidity} is true for higher-dimensional spheres $S^{n-1}$ with $n\leq k^*-1$, where $k^*$ is the critical dimension.
This is somehow connected to our current knowledge and proofs about $k^*$. 
From \cite{CJK04}, a sufficient condition for instability is the existence of strict subsolutions to the linearized equation of \eqref{eq:PDE} (also see Proposition 2.1, 2.2 in \cite{JS15}, and Proposition \ref{prop:unstab_sym_homoge_var}).
Then in \cite{JS15}, they proved $k^*\geq 5$ by constructing such strict subsolutions explicitly on $\mathbb{R}^3$ (which recovered the result of \cite{CJK04}) and $\mathbb{R}^4$ (which is new).
The construction used $u$ as input and hence can be generalized to symmetric versions with respect to $\Gamma$, provided $u$ is also symmetric. 
Therefore, it is possible to upgrade Theorem \ref{thm:rigidity} by our argument if the construction of such subsolutions can be generalized to higher dimensions.

\subsection{Structure of the Article}
    In Section \ref{sec:prelim}, we provide preliminaries regarding the one-phase problem and Ricci limit spaces and introduce the singular strata in various contexts. 
    In Section \ref{sec:stab_under_sym_var}, we introduce the symmetry tools and prove several necessary conditions of stability under symmetric variations.
    In Section \ref{sec:proof_main}, we prove the main results.

\subsection*{Acknowledgment}
The research leading to these results is part of a project that has received funding from the European Research Council (ERC) under the European Union's Horizon 2020 research and innovation programme (grant agreement No 101001677).

\section{Preliminaries}\label{sec:prelim}
Throughout this article, a \emph{metric measure space} is a triple $(X,d,\mu)$, where $(X,d)$ is a complete separable metric space, and the reference measure $\mu$ is a Radon measure on $X$ such that $\supp \mu=X$.
Within our framework, all metric spaces involved are assumed to be proper.
Unless specified, a domain in a metric space is a nonempty open set.
We write $B_r(x)$ for the open ball centered at $x\in X$ of radius $r>0$.
Given $A\subset X$, we let $1_A$ denote its characteristic function and let $B_r(A)$ denote its $r$-tubular open neighborhood.
We denote by $\mathcal{H}^n$ the $n$-dimensional Hausdorff measure, which will be the reference measure $\mu$ in most of our considerations, while the dimension $n$ depends on context.
We denote by $\dim_H$ the Hausdorff dimension.
\par
Here we recall the definition of metric cones (see e.g.\@ \cite{BBI01}):
\begin{definition}\label{def:metric_cone}
    Let $(X,d)$ be a metric space.
    The metric cone over $X$, denoted by $C(X)$, is the topological space $\left([0,\infty)\times X\right)/\left(\{0\}\times X\right)$ equipped with the metric
    \begin{equation*}
        d_{C(X)}^2((r_1, x_1), (r_2, x_2))=r_1^2+r_2^2-2r_1r_2\cos(\min(d_X(x_1, x_2),\pi)).
    \end{equation*}
    The point $[\{0\}\times X]\in C(X)$ is called the tip of the metric cone.
\end{definition}
One can check that when $(X,g_X)$ is a complete Riemannian manifold, this definition agrees with the Riemannian distance from the warped product metric tensor $dr^2+r^2g_X$.

\subsection{Ricci Limit Space}
This subsection is adapted from \cite{CF24}.
\par
The starting point of the theory of Ricci limit spaces is the notion of (pointed) Gromov--Hausdorff convergence.
See e.g.\@ \cite{BBI01, Vil09} for an overview.
\begin{definition}[Pointed Gromov--Hausdorff convergence]\label{def:pGH}
     A sequence of pointed metric spaces $(X_j , d_j , x_j )$ is said to converge in the pointed Gromov--Hausdorff (pGH) topology to a pointed metric space $(X, d, x)$ if there exists a separable metric space $(Z, d_Z )$ and isometric embeddings $i_j : X_j \to  Z$ and $i: X \to Z$ such that the following holds:
     For any $\epsilon>0$ and $R>0$, there exists $\bar{j}\in \mathbb{N}$ such that for every $j \geq  \bar{j}$, we have
 \begin{equation*}
        \begin{cases}
             d_Z(i(x), i_j(x_j))<\epsilon,\\
             i(B^{X}_R(x))\subset B^Z_\epsilon \left(i_j(B^{X_j}_{R}(x_j))\right)\quad \text{and} \quad 
             i_j(B^{X_j}_R(x_j))\subset B^Z_\epsilon \left(i(B^{X}_{R}(x))\right).
        \end{cases}
\end{equation*}
     We denote the pointed Gromov--Hausdorff convergence by $X_j\xrightarrow{\text{pGH}}X$ and say that $Z$ is the space realizing the convergence.
\end{definition}

We now define the spaces that we consider in our theorems.
\begin{definition}[Noncollapsed limits of manifolds with lower/two-sided bounds on the Ricci curvature]\label{def:ric_lim}
     Let $(M_j , d_j , x_j )$ be a sequence of pointed Riemannian manifolds of fixed dimension $n \geq 2$ such that
     \begin{equation}\label{eq:ric_lower_bd}
         \Ric_{M_j}\geq -(n-1).
     \end{equation}
     Suppose there exists a pointed metric space $(X, d, x)$ such that $M_j\xrightarrow{\text{pGH}} X$.
     Then we say that $X$ is a \emph{Ricci limit space}.
     If condition \eqref{eq:ric_lower_bd} is strengthened to
    \begin{equation*}
        |\Ric_{M_j}|\leq n-1,
    \end{equation*}
    we say that $X$ is a \emph{limit of manifolds with two-sided bounds on the Ricci curvature}.
    Moreover, in either case, if
    \begin{equation}\label{eq:noncollapsed}
        \vol(B_1(x_j))\geq v>0
    \end{equation}
    for some constant $v$, we say that $X$ is \emph{noncollapsed}.
\end{definition}

Recall that by Bishop--Gromov's volume comparison and Gromov's precompactness theorem, assumption \eqref{eq:ric_lower_bd} guarantees that any collection of such manifolds admits a converging sequence in the pGH-topology.
This is one of the motivations to study Ricci limit spaces.
\par
Ricci limit spaces were studied extensively in \cite{CC97, CC00a, CC00b}.
In particular, \cite[Theorem 5.9]{CC97} showed that the noncollapsed assumption \eqref{eq:noncollapsed} forces the Hausdorff dimension of the
limit space and that of the approximating sequence to be the same.
Notice that the curvature bounds above can be relaxed to any real numbers of the same sign by rescaling the metrics.
\par

Let us recall the definition of tangent spaces in the setting of metric spaces.
They describe the infinitesimal behavior of such spaces.

\begin{definition}[Tangent space of a metric space at a point]\label{def:tang_space}
    Let $(X, d)$ be a metric space and let $x \in X$.
    We define the space of tangent spaces at $x$, denoted by $\Tan(X, x)$, to be the set of all pointed metric spaces $(Y , d_Y , y)$ such that 
    \begin{equation*}
        (X, d/r_j , x) \xrightarrow{\text{pGH}} (Y,d_Y,y)
    \end{equation*}
    for some sequence $r_j\in (0,1)$ with  $r_j \to 0$.
\end{definition}
In the case of Ricci limit spaces, by the same precompactness theorem above, the set of tangent spaces is always non-empty. Moreover, \cite[Theorem 5.2]{CC97} showed that for noncollapsed Ricci limit spaces, all tangent spaces are metric cones, and the point $y$ is the cone tip.
\par
In the terminology of geometric analysis, these metric cones are $0$-symmetric objects and can be further classified based on the amount of symmetries they have, in the sense of how many Euclidean factors they split off.
This motivates the following stratification of Ricci limit spaces.

\begin{definition}[Singular Strata of Ricci limit spaces]\label{def:ric_lim_strata}
    Let $X$ be a noncollapsed Ricci limit space.
    For $k \in \mathbb{N}_0$, the $k$-singular stratum $\mathcal{S}^k(X)$ collects those $x\in X$ such that any $(Y,d_Y,y)\in \Tan(X,x)$ splits off at most $k$ Euclidean factors.
    That is to say, any $(Y,d_Y,y)\in \Tan(X,x)$ cannot be isometric to $\mathbb{R}^{k+1}\times C(Z)$ for any metric cone $C(Z)$.
\end{definition}

By definition, we clearly have the inclusions
\begin{equation*}
    \mathcal{S}^0(X)\subset \mathcal{S}^1(X)\subset \cdots \subset \mathcal{S}^n(X)=X,
\end{equation*}
when $X$ is $n$-dimensional.
For those points achieving the maximal symmetry, we define
\begin{definition}\label{def:ric_lim_reg_sing}
    The regular set $\mathcal{R}(X)$ collects those $x\in X$ such that 
    \begin{equation*}
        \Tan(X,x)=\{(\mathbb{R}^n,d_{\text{euc}},0)\}.
    \end{equation*}
    The singular set is defined as the complement, $\mathcal{S}(X):=X\backslash\mathcal{R}(X)$.
\end{definition}

Here, we collect some key results about noncollapsed Ricci limit spaces.
By the rigidity of volume comparison and \cite{colding1997ricci}, we have $x\in \mathcal{R}(X)$ if and only if
\begin{equation*}
    (\mathbb{R}^n,d_{\text{euc}},0)\in \Tan(X,x).
\end{equation*}
Hence we have  $\mathcal{S}(X)=\mathcal{S}^{n-1}(X)$.
Another rigidity result \cite[Theorem 6.2]{CC97} showed that 
\begin{equation}\label{eq:ric_lim_strata_n-2=n-1}
    \mathcal{S}^{n-1}(X)\backslash\mathcal{S}^{n-2}(X)=\emptyset.
\end{equation}
On the other hand, using Federer's dimension reduction, it was proved in \cite[Theorem 4.7]{CC97} that
\begin{equation*}
    \dim_H \mathcal{S}^k(X)\leq k.
\end{equation*}
Combining these, they concluded the dimension bound \cite[Theorem 6.1]{CC97}: 
\begin{equation}\label{eq:Ric_sing_dim_bd}
    \dim_H\mathcal{S}(X)\leq n-2.
\end{equation}
\par
With the stronger two-sided bounds on the Ricci curvature, the result above can be strengthened:
\begin{theorem}[{\cite[Theorem 5.12]{CN15}}]\label{thm:dim_bd_twosided_ricci}
    Let $(X, d, x)$ be a noncollapsed limit of manifolds with two-sided Ricci curvature bounds of dimension $n \geq 2$.
    Then 
    \begin{equation*}
        \mathcal{S}(X)\backslash\mathcal{S}^{n-4}(X)=\emptyset.
    \end{equation*}
    In particular, we have $\dim_H \mathcal{S}(X)\leq n-4$ (\cite[Theorem 1.4]{CN15}).
\end{theorem}

As for the next lower stratum, we have the following structure theorem, which is a key ingredient for the proof of Theorem \ref{thm:dim_bound}.
\begin{theorem}[{\cite[Theorem 1.16]{CJN21}}]\label{thm:twosided_ric_n-4_rig}
    Let $(X, d, x)$ be a noncollapsed limit of manifolds with two-sided bounds on the Ricci curvature of dimension $n\geq 2$.
    Then for any $x \in X\backslash\mathcal{S}^{n-5}(X)$, there exists a tangent space at $x$ which is isometric to $\mathbb{R}^{n-4}\times C(S^3/\Gamma)$, where $\Gamma\subset O(4)$ is a discrete group acting freely.
\end{theorem}
\begin{proof}
    The proof is already given in \cite[Theorem 2.8]{CF24} and included here for completeness.
    Let $x\in X\backslash\mathcal{S}^{n-5}(X)$.
Then, there exists a tangent space to $X$ at $x$ of the form $\mathbb{R}^{n-4}\times C(Y)$ for some metric space $Y$.
By \cite{CN15} and \cite{And90}, it follows that $Y$ is a $3$-dimensional smooth Riemannian manifold and that the Ricci curvature of $C(Y)$ outside of the tip vanishes.
This implies that $Y$ has constant sectional curvature equal to $1$, so that (see e.g.\@ \cite[Subsection 10.2]{Pet06}) it is a quotient $S^3/ \Gamma$, where $\Gamma\subset O(4)$ is a discrete group acting freely.
\end{proof}
Finally, we mention the following regularity result about $\mathcal{R}(X)$ for such spaces:
\begin{theorem}[\cite{And90, CC96}]\label{thm:twosided_ricci_R(X)_holder}
    Let $(X, d, x)$ be a noncollapsed limit of manifolds with two-sided bounds on the Ricci curvature of dimension $n\geq 2$.
    Then for every $\beta\in (0,1)$, $\mathcal{R}(X)$ is an open $C^{1,\beta}$ Riemannian manifold.
\end{theorem}
We point out that this theorem involves the $\epsilon$-regularity results from \cite{And90}, which are not valid for general noncollapsed Ricci limit spaces.
In the latter case, there are examples where $\mathcal{S}(X)$ is dense (see e.g. \cite{otsu1994riemannian}).
Still, the Reifenberg-type results in \cite{CC97} showed that outside \emph{a subset of $\mathcal{S}(X)$}, $X$ is locally bi-H\"older to a Riemannian manifold.

\subsection{One-phase Free Boundary Problem}
In this subsection, we review some key results about the one-phase free boundary problem, starting with the Euclidean setting.
We refer to \cite{CS05,Vel23} for an overview of the topic.
The non-smooth setting will be discussed in Subsection \ref{subsec:FBP_nonsmooth}. 
\begin{theorem}[Existence and fundamental regularity, {\cite{AC81}}]\label{thm:Euc_one-phase_fund}
Consider the one-phase problem \eqref{1-phase_Ber} on a bounded domain $\Omega \subset \mathbb{R}^n$.
Assume $\|Q\|_{C^0(\Omega)} + \|Q^{-1}\|_{C^0(\Omega)} \leq \Lambda$.
Then we have the following:
\begin{enumerate}[topsep=0pt, ]
    \item\textbf{Existence and Euler-Lagrange equation}:
   Minimizers exist (not uniquely, in general).
   For any minimizer $u$, the set $\{u>0\}\cap \Omega$ has locally finite perimeter, and we have
    \[
    \Delta u = Q\,\mathcal H^{n-1}\llcorner \partial\{u>0\}\cap \Omega
    \]
    in the sense of distributions.
    In particular (also see \cite{Caf89}), $u$ solves the boundary value problem \eqref{eq:PDE} with the free boundary in the viscosity sense.
    \item\textbf{Lipschitz regularity.}
    For $x \in \Omega$ and any ball $B_{2r}(x)\subset \Omega$, we have
    \[
    \|\nabla u\|_{L^\infty(B_r(x))} \leq C(n,\Omega,\Lambda,v_0,r).
    \]
    
    \item \textbf{Non-degeneracy.}
    For any $x\in \{u>0\}$ and any ball $B_{2r}(x)\subset \Omega$, we have
    \[
        \|u\|_{L^\infty(B_r(x))}\geq C(n,\Omega,\Lambda,v_0)r.
    \]
    
    \item \textbf{Regularity of the free boundary:}
    The free boundary $\partial\{u>0\}\cap \Omega$ is a $C^{1,\alpha}$-manifold away from a relatively closed subset $\mathcal{S}(u)$ with $\mathcal{H}^{n-1}(\mathcal{S}(u))=0$.
\end{enumerate}
\end{theorem}
Based on the last item above, one defines the regular set of the free boundary, $\mathcal{R}(u)$, as the set of points where the free boundary is locally a $C^{1,\alpha}$-manifold.
Then its complement is exactly the $\mathcal{S}(u)$ above, defined as the singular set.
\par
\cite[4.7]{AC81} also introduced the blowups of minimizers: for $x\in \pt\{u>0\}\cap \Omega$ and $y\in \mathbb{R}^n$,
\begin{equation}\label{eq:blowups}
    u_{x}(y):=\lim_{t_i\to 0} \frac{1}{t_i}u(x+t_iy),
\end{equation}
where $\{t_i\}$ is some sequence converging to $0$.
With their convergence results, they showed that the blowups are well-defined global minimizers on $\mathbb{R}^n$ with constant $Q=Q(x)$. 
Furthermore, their $\epsilon$-regularity theorem \cite[Theorem 8.1]{AC81} showed that $\mathcal{R}(u)$ contains exactly those points where there is a linear blowup as in \eqref{eq:linear_sol}.
\par
More systematic study of $\mathcal{S}(u)$ begins with \cite{Wei99}, who introduced the Weiss monotonicity formula:
\begin{theorem}[\cite{Wei99}, Theorem 1.2]\label{thm:orginal_weiss}
    Let  $u$ be a minimizer of the one-phase problem \eqref{1-phase_Ber} on a domain $\Omega\subset \mathbb{R}^n$, with $Q \equiv 1$. Fix $x\in \pt\{u>0\}\cap \Omega$.
    For $r\in(0,d(x, \pt \Omega))$, define
    \begin{equation*}
        W(r):=r^{-n}\int_{B_r(x)}(|\nabla u|^2 +1_{u>0})\dd x 
        -r^{-n-1}\int_{\pt B_r(x)}u^2 \dd{\mathcal{H}^{n-1}}.
    \end{equation*}
    Then we have
    \begin{equation*}
        W'(r)=\int_{\pt B_r(x)} \frac{|u(y)-(y-x)\cdot \nabla u(y)|^2}{|y-x|^{n+2}}\dd{\mathcal{H}^{n-1}}(y).
    \end{equation*}
\end{theorem}
This shows that the Weiss quantity $W$ is increasing in $r$, and is constant if and only if $u$ is $1$-homogeneous with respect to the fixed point $x$.
In particular, this implies that blowups and blowdowns of minimizers are $1$-homogeneous (\cite[Theorem 2.8]{Wei99}), which is conjectured in \cite{AC81}.
In the terminology of geometric analysis, these $1$-homogeneous minimizers are $0$-symmetric objects and can be further classified based on the amount of symmetries they have, in the sense of how many invariant orthogonal directions they have.
This opens the door towards Federer's dimension-reduction argument: In \cite[Theorem 4.5]{Wei99}, it was proved that
\begin{equation}\label{eq:dim_bd_euc_1-phase_sing}
    \dim_H \mathcal{S}(u)\leq n-k^*,
\end{equation}
where $k^*$ is the critical dimension in Definition \ref{def:crit_dim}.
With later works \cite{CJK04, JS15, DJ09} about the existence/nonexistence of nonlinear homogeneous global minimizers, we now know $k^*\in \{5,6,7\}$, although the exact value is still an open question.
We also mention the work \cite{EE18} where (quantitative) singular strata for the one-phase problem in the Euclidean setting were introduced to show finer properties of the singular set, such as rectifiability.
\par
Finally, we mention the following Liouville-type theorem, which is another consequence of the Weiss monotonicity formula: 
\begin{corollary}[\cite{Wan12}]\label{cor:global_min_linear}
    If $n$ is less than the critical dimension $k^*$, then any nontrivial (potentially inhomogeneous) global minimizer of the one-phase problem with constant $Q$ on $\mathbb{R}^n$ is a linear solution as in \eqref{eq:linear_sol}, up to translation.
\end{corollary}

\subsubsection{One-phase Problem in Non-smooth Setting}\label{subsec:FBP_nonsmooth}
In \cite{CZZ22}, the authors extended the study of the one-phase problem \eqref{1-phase_Ber} to the class of non-smooth spaces known as $\mathsf{RCD}(K,N)$ spaces.
These are metric measure spaces that have, in a synthetic sense, Riemannian (``R") Ricci curvatures (``C") bounded from below by $K$ and the dimension (``D") bounded from above by $N$, where $K\in\mathbb{R}$ and $N\in [1,\infty]$.
We refer to the survey \cite{ambrosio2018calculus} and the references therein for an overview of the $\mathsf{RCD}$ theory, and point out that it includes the Ricci limit spaces as in Definition \ref{def:ric_lim} (see \cite{Gig15}).
In particular, all fundamental results in \cite{CZZ22} are valid in our case, and we recall some of them in this subsection.
Throughout this article, we only consider noncollapsed $\mathsf{RCD}$ spaces without boundary (see \cite{KM21} for related definitions and stability results):
\par
\begin{definition}[\cite{DG18}]
    We say that an $\mathsf{RCD}(K,n)$ space $(X,d,\mu)$ is noncollapsed if $\mu=\mathcal{H}^n$.
\end{definition}
We remark that $n$-dimensional noncollapsed Ricci limit spaces in Definition \ref{def:ric_lim} are indeed noncollapsed $\mathsf{RCD}(-(n-1), n)$ spaces (see \cite{CC97,DG18}).
\par
For the one-phase problem on these noncollapsed $\mathsf{RCD}$ spaces, we have the following results, partially generalizing Theorem \ref{thm:Euc_one-phase_fund}:

\begin{theorem}[{\cite[Subsection 1.2, Theorem 6.1]{CZZ22}}]\label{thm:RCD_FBP_fund}
Consider the one-phase problem \eqref{1-phase_Ber} on a bounded domain $\Omega$ in a noncollapsed $\mathsf{RCD}(K, n)$ space $X$, with $\|Q\|_{C^0(\Omega)} + \|Q^{-1}\|_{C^0(\Omega)} \leq \Lambda$.
Then Items 1, 2, and 3 in Theorem \ref{thm:Euc_one-phase_fund} hold with constants depending also on $K$.
\end{theorem}
\begin{remark}
Item 1 of the Theorem above involves the notion of sets of finite perimeter in $\mathsf{RCD}$ spaces. 
See e.g. \cite{ABS19_finiteperim, BPS22, BPS23}.
\end{remark}

In contrast to Item 4 in Theorem \ref{thm:Euc_one-phase_fund}, the regularity of the free boundary in the noncollapsed $\mathsf{RCD}(K,n)$ setting is more subtle since singularities may come from the ambient space itself.
This is manifested in the blowups of minimizers, which are now functions defined on tangent spaces that are not necessarily $\mathbb{R}^n$. 
In order to perform blowup analysis in this context, we first need the following notion about the convergence of metric measure spaces, which extends Definition \ref{def:pGH}.

\begin{definition}[Pointed measured Gromov--Hausdorff convergence]\label{def:pmGH}
A sequence of pointed metric measure spaces $(X_j , d_j , \mu_j , x_j)$ is said to converge in the pointed measured Gromov--Hausdorff (pmGH) topology to a pointed metric measure space $(X, d, \mu, x)$ if both
\begin{enumerate}[topsep=0pt, ]
    \item $X_j\xrightarrow{\text{pGH}}X$ as in Definition \ref{def:pGH};
    \item Using the same notations from Definition \ref{def:pGH}, the pushed-forward measures $(i_j)_\#\mu_j\rightharpoonup i_\#\mu$ weakly in duality of continuous functions on $Z$ with bounded support. 
\end{enumerate}
We denote this convergence by $X_j \xrightarrow{\text{pmGH}} X$ and say that $Z$ is the space realizing the convergence.
\end{definition}
We mention that when defining noncollapsed Ricci limit spaces in Definition \ref{def:ric_lim}, we actually have pmGH convergence if all $X_j$ and $X$ are equipped with their corresponding $\mathcal{H}^n$ (see \cite[Theorem 1.2]{DG18}).
In particular, any tangent space (as in definition \ref{def:tang_space})
\begin{equation*}
    (Y,d_Y, \mathcal{H}^n_{d_Y}, y)\in \Tan(X,x)
\end{equation*}
is the pmGH limit of the sequence $X_j:=(X, d_j:=d/r_j, \mathcal{H}^n_{d_j}, x)$ some some $r_j\to 0$.
From now on, we will often omit the reference measures $\mathcal{H}^n$.
\begin{remark}\label{rmk:rcd_tang_cone}
    More generally, such tangent spaces can be defined in the same way for noncollapsed $\mathsf{RCD}(K,n)$ spaces.
By \cite{Stu06}, \cite[Proposition 2.8]{DG18}, these tangent spaces exist and are noncollapsed $\mathsf{RCD}(0,n)$ cones, which by \cite[Corollary 1.3]{Ket15} are exactly those metric measure cones over noncollapsed $\mathsf{RCD}(n-2,n-1)$ spaces with diameter no greater than $\pi$.
Furthermore, the singular set $\mathcal{S}(X)$ and strata $\mathcal{S}^k(X)$ can also be defined for noncollapsed $\mathsf{RCD}$ spaces, as in Definition \ref{def:ric_lim_reg_sing} and \ref{def:ric_lim_strata} for Ricci limit spaces, and it is shown in \cite[Theorem 1.8]{DG18} that $\dim_H \mathcal{S}(X)\leq n-2$ (recall we assume that these spaces are boundary-free; in particular, the tangent cones have no boundary).
\end{remark}

In addition to the convergence of spaces, we also need the notion about uniform convergence of functions along converging spaces:
\begin{definition}
Suppose $X_j\xrightarrow{\text{pmGH}} X$ is realized in $Z$.
Consider a sequence of measurable functions $f_j$ on $B_R(x_j)\subset X_j$ and let $f$ be a measurable function on $B_R(x)$.
We say that $f_j$ converges to $f$ uniformly over $B_R(x_j)$ if for any $\epsilon>0$, there exist $N(\epsilon) \in \mathbb{N}$ and $\delta(\epsilon)$ such that 
    \[
\sup_{\substack{z \in B_R(x_j),w \in B_R(x) \\ d_Z(i_j(z), i(w)) < \delta}}
\left| f_j(z) - f(w) \right| < \varepsilon, \quad \forall k \ge N.
\]
We denote this convergence by $f_j\xrightarrow{\text{uniform}} f$.
\end{definition}

We are now ready to define blowups of functions:

\begin{definition}[Blowups of functions at a point]\label{def:tang_func}
    Let $u$ be a measurable function on a metric measure space $(X,d,\mathcal{H}^n)$ and let $x \in X$.
    We define the space of blowups of $u$ at $x$, denoted by $\Tan(X,u, x)$, to be the set of all quadruples $(Y,d_Y, u_x , y)$ such that 
    \begin{align*}
        (X, d_j:=d/r_j , \mathcal{H}^n_{d_j}, x) &\xrightarrow{\text{pmGH}} (Y,d_Y,\mathcal{H}^n_{d_Y}, y)\\
        u_j:=u/r_j &\xrightarrow{\text{uniform}} u_x,
    \end{align*}
    for some sequence $r_j\in (0,1)$ with  $r_j \to 0$.
\end{definition}

For blowups of minimizers at the free boundary, the compactness results \cite[Theorem 7.1]{CZZ22} imply the following:
\begin{lemma}\label{lem:RCD_blowup_prop}
    Let $u$ be a minimizer of the one-phase problem \eqref{1-phase_Ber} in the noncollapsed $\mathsf{RCD}$ setting, and $x\in \pt\{u>0\}\cap \Omega$.
    Then $\Tan(X,u,x)$ is never empty.
    Furthermore, for any $(Y,d_Y, u_x, y)\in \Tan(X,u,x)$, $u_x$ is a global minimizer on the tangent cone $Y$ with constant $Q=Q(x)$ and $y\in \pt\{u_x>0\}$.
\end{lemma}

    Here is another big difference from the Euclidean case.
    Due to the lack of a general monotonicity formula on the ambient $X$, it is unknown whether the blowups $u_x$ are $1$-homogeneous on the corresponding tangent cone $Y$.
    Nonetheless, by the Weiss-type formula on $\mathsf{RCD}$ cones \cite[Appendix A]{CZZ22}, the double-blowup $(u_x)_y$ will be $1$-homogeneous on $Y$ (notice that any blowup of a metric measure cone at its tip is itself).
    A similar situation has been observed in \cite{FMS25} for perimeter minimizers in $\mathsf{RCD}$ spaces: It is unknown whether blowups of perimeter minimizing sets are perimeter minimizing cones in the corresponding tangent cone.
    Such situations create issues for the classical dimension reduction arguments, which are addressed in \cite[Theorem 4.5]{FMS25}.
    Their argument is adapted in Lemma \ref{lem:dim_bd_strata_two_asp} for our one-phase problems. 
\par    
Although the blowups may not be $1$-homogeneous, they can still be classified based on the amount of symmetries they have:
\begin{definition}[$k$-symmetric blowups of minimizers]\label{def:k_sym_both_asp}
    Let $u$ be a minimizer of the one-phase problem \eqref{1-phase_Ber} in the noncollapsed $\mathsf{RCD}$ setting.
    For $x\in \pt\{u>0\}\cap \Omega$, we say $(Y,d_Y, u_x , y)\in \Tan(X,u,x)$ is $k$-symmetric if both 
    \begin{enumerate}[topsep=0pt, ]
        \item the metric measure cone $Y \cong \mathbb{R}^k\times C(Z)$ for some metric measure cone $C(Z)$;
        \item Using the identification above, $u_x$ is invariant in the $\mathbb{R}^k$ factor. 
    \end{enumerate}
\end{definition}

For such $k$-symmetric blowups $u_x$, we further have
\begin{lemma}[{\cite[Lemma 8.13]{CZZ22}}, {\cite[Lemma 3.2]{Wei99}}]\label{lem:restrict_min_is_min}
    Using the identification above, the section
    \begin{equation*}
        u_x|_{\{0^k\}\times C(Z)}
    \end{equation*}
    is a global minimizer of the one-phase problem on $C(Z)$ with constant $Q=Q(x)$.
\end{lemma}
\begin{remark}\label{rmk:sec_min_min_corresp}
        In fact, we have the following general fact(see e.g.\@ \cite[Lemma 10.10]{Vel23}): 
        Let $Z$ be a noncollapsed $\mathsf{RCD}$ space.
        Let $v\in W^{1,2}_{\text{loc}}(Z)$ be a nonnegative function, and define $\tilde{v}:\mathbb{R}\times Z\to \mathbb{R}$ by
        \begin{equation*}
        \tilde{v}(t,z)=v(z).
        \end{equation*}
        Assuming $Q$ is constant.
        Then $v$ is a global minimizer on $Z$ if and only if $\tilde{v}$ is a global minimizer on $\mathbb{R}\times Z$.
\end{remark}

With the notion of symmetry in mind, we now define the singular strata of $\pt\{u>0\}\cap \Omega$:
\begin{definition}[Singular strata of free boundary]\label{def:strata_both_asp}
    For $k\in \mathbb{N}_0$, the $k$-singular stratum $\mathcal{S}^k(u)$ collects those points in $\pt\{u>0\}\cap \Omega$ such that no blowups are $(k+1)$-symmetric in the sense of Definition \ref{def:k_sym_both_asp}.
    That is, any blowup at these points is at most $k$-symmetric.
\end{definition}
Such strata can roughly be viewed as a combination of the strata in Definition \ref{def:ric_lim_strata} and the strata for one-phase problems in \cite[Definition 1.8]{EE18}.
Clearly, we have the following inclusion:
\begin{equation}\label{eq:dual_strata_inclu}
    \mathcal{S}^0(u) \subset \mathcal{S}^1(u) \subset \ldots \subset \mathcal{S}^{n-1}(u)=\pt\{u>0\}\cap\Omega,
\end{equation}
where the last equality follows from Item 3: non-degeneracy in Theorem \ref{thm:RCD_FBP_fund}.
For those points attaining the maximal symmetry, we define
\begin{definition}\label{def:RCD_FBP_reg_sing}
    The regular set $\mathcal{R}(u)$ collects those $x\in \pt\{u>0\}\cap \Omega$ such that 
    \begin{equation*}
        \Tan(X,u,x)=\{(\mathbb{R}^n,d_{\text{euc}},h,0)\},
    \end{equation*}
    where $h$ is the linear solution in \eqref{eq:linear_sol} with $\nu=e_n$; notice that all linear solutions are identified as the same here since we mod out rotations on $\mathbb{R}^n$.
    The singular set is defined as the complement, $\mathcal{S}(u):=\pt\{u>0\}\cap \Omega \backslash\mathcal{R}(u)$.
\end{definition}
We remark that by the $\epsilon$-regularity in \cite[Theorem 8.6]{CZZ22}, we have $x\in \mathcal{R}(u)$ if and only if
\begin{equation*}
    (\mathbb{R}^n,d_{\text{euc}},h,0)\in \Tan(X,u,x).
\end{equation*}
Together with $\mathcal{S}(X)=\mathcal{S}^{n-2}(X)$ in \eqref{eq:ric_lim_strata_n-2=n-1}, we then have
\begin{equation}\label{eq:dual_sing_equal_n-2_strata}
    \mathcal{S}(u)= \mathcal{S}^{n-2}(u).
\end{equation}
\par
We also point out that the (infinitesimal) symmetry of $u$ implies that of the free boundary at $x$.
For example, if $x\in \mathcal{R}(u)$, then blowups of $\{u>0\}\cap \Omega$ at $x$ can only be
\begin{equation*}
    (\mathbb{R}^n, d_{\text{euc}},\mathbb{R}^{n-1}\times [0,+\infty), 0)
\end{equation*}
(see \cite[Definition 2.13]{CF24} for the definition of blowups of sets of locally finite perimeter).
Likewise, similar symmetry results hold for points in $\mathcal{S}^k(u)$.

\par
With the preparations above, we can now state the regularity results of the free boundary in the noncollapsed $\mathsf{RCD}$ setting:

\begin{theorem}[{\cite[Theorem 1.9, Corollary 1.11]{CZZ22}}]\label{thm:RCD_free_bdy_reg}
Suppose $u$ is a minimizer of the one-phase problem \eqref{1-phase_Ber} in the noncollapsed $\mathsf{RCD}(K,n)$ setting, with $\|Q\|_{C^0(\Omega)} + \|Q^{-1}\|_{C^0(\Omega)} \leq \Lambda$.
Then we have
\begin{enumerate}[topsep=0pt, ]
    \item \textbf{Dimension bound on the singular set}: $\dim_{H}\mathcal{S}(u) \leq n-3$.
    \item \textbf{Manifold structure}: There is a relatively open set $\mathcal{O}\subset \pt\{u>0\}\cap \Omega$ containing $\mathcal{R}(u)$ that is locally bi-H\"older homeomorphic to a $(n-1)$-dimensional Riemannian manifold.
\end{enumerate}
\end{theorem}

By Example \ref{ex:RCD_dense_codim3} below, we see that the dimension bound in Item 1 above is sharp.
It also shows that $\mathcal{S}(u)$ can be dense in $\pt\{u>0\}\cap \Omega$ and, in particular, may not be closed. 
These should be compared to $\dim_H\mathcal{S}(X)\leq n-2$ (see Remark \ref{rmk:rcd_tang_cone}), and to Item 4 in Theorem \ref{thm:Euc_one-phase_fund}.
We remark that the same dimension bound was proved in \cite{MS25} for the singular set of perimeter minimizing sets in noncollapsed $\mathsf{RCD}$ spaces.
In some sense, these bounds say that these $(n-1)$-dimensional objects such as $\pt\{u>0\}\cap \Omega$ tend to avoid $\mathcal{S}(X)$.

\begin{example}[Dense codimension-$3$ singularities in free boundary, {\cite[Remark 1.12]{CZZ22}}]\label{ex:RCD_dense_codim3}
By \cite{otsu1994riemannian}, there exists a two-dimensional Alexandrov space $Y$ with nonnegative curvature and without boundary, such that $\mathcal{S}(Y)$ is dense.
On the bounded domain $(-1,1)\times Y\subset \mathbb{R}\times Y$, one can check that the function $f(t,y):=t_+$ is a minimizer of the one-phase problem with $Q=1$.
Its free boundary is $\{0\}\times Y$, which has dense singularities.
\end{example}

The proof of the dimension bound in Theorem \ref{thm:RCD_free_bdy_reg} followed the same idea as those behind \eqref{eq:Ric_sing_dim_bd} and \eqref{eq:dim_bd_euc_1-phase_sing}, using Federer's dimension reduction (Lemma \ref{lem:dim_bd_strata_two_asp} below) and a rigidity result about minimizers on $2$-dimensional circular cones (Lemma \ref{lem:FBP_2dim_cone} below).
The rigidity result implies that 
\begin{equation*}
    \mathcal{S}(u)=\mathcal{S}^{n-3}(u),
\end{equation*}
while the dimension reduction gives
\begin{lemma}[{\cite[Theorem 8.12]{CZZ22}}]\label{lem:dim_bd_strata_two_asp}
With the same setting as in Theorem \ref{thm:RCD_free_bdy_reg}, we have
    \begin{equation*}
        \dim_H \mathcal{S}^k(u)\leq k.
    \end{equation*}
\end{lemma}

As we mentioned after Lemma \ref{lem:RCD_blowup_prop}, it is unknown whether the blowups are $1$-homogeneous on the corresponding tangent cone.
This creates issues for the classical dimension reduction, which we address in Appendix \ref{apndx:dim_bound_strata}.
\par
We end this subsection with the aforementioned rigidity result:
\begin{lemma}[{\cite{AL15}}]\label{lem:FBP_2dim_cone}
    Let $u$ be a global minimizer of the one-phase problem on the $2$-dimensional circular cone $C(S^1_r)$ with constant $Q$, where $S^1_r$ is a circle of radius $r\leq 1$.
    If $r<1$, then the tip is not in the free boundary.
\end{lemma}
We remark that by \cite{KL16}, any noncollapsed $\mathsf{RCD}(0,2)$ cones (without boundary) are exactly these circular cones $C(S^1_r)$.

\section{Stability Under Symmetric Variation}\label{sec:stab_under_sym_var}

Our strategy to prove the main theorems follows \cite{CF24} by analyzing the blowups of minimizers at points in the singular strata.
Based on Lemma \ref{lem:RCD_blowup_prop} and \ref{lem:restrict_min_is_min}, it suffices to consider global minimizers of the one-phase problem with constant $Q$ on certain metric cones.
Such cones, in our case, arise as symmetric quotients of the Euclidean space (Theorem \ref{thm:twosided_ric_n-4_rig}).
This motivates the consideration of the one-phase problem in a symmetric setting.
\par
In Subsection \ref{subsec:sym_obj}, we introduce some terminology and basic lemmas for our symmetric setting, which are modified from \cite{CF24}.
These relate global minimizers on the metric cones to symmetric minimizers on $\mathbb{R}^n$ under symmetric variations, for a general dimension $n\geq 3$.
In Subsection \ref{subsec:cond_stab_sym_var}, we prove several necessary conditions for stability under symmetric variations on $\mathbb{R}^n$, which are modified from \cite{CJK04, JS15}.
The case $n=4$ will be used in the next section to show that any stable symmetric $1$-homogeneous solution on $\mathbb{R}^4$ under symmetric variations is a linear solution \eqref{eq:linear_sol}. 

\subsection{Symmetry Tools}\label{subsec:sym_obj}
Let us fix some notations.
In this section, $n\in \mathbb{N}$, $n\geq 3$, and  $\Gamma\subset O(n)$ is a discrete group of isometries acting freely on $S^{n-1}$ and $\mathbb{R}^n\backslash \{0\}$, with $|\Gamma|=l$.
We denote by $\pi : \mathbb{R}^n \to \mathbb{R}^n/ \Gamma$ the projection map.
Then $\pi|_{\mathbb{R}^n \backslash \{0\}}$ is a covering map of $(\mathbb{R}^n\backslash\{0\})/ \Gamma$, and hence a local isometry.
The setup here is general, although we will only use the case $n=4$ for our main theorems.
Notice that the space $\mathbb{R}/\Gamma\cong C(S^{n-1}/\Gamma)$ has a natural metric cone structure and connects to Theorem \ref{thm:twosided_ric_n-4_rig}.
Throughout this subsection, we only consider the one-phase problems with constant $Q=1$.
\par
\begin{definition}[Cover chart]
    We say that an open set $U \subset (\mathbb{R}^n \backslash\{0\})/ \Gamma$ is a cover chart if its preimage through
$\pi$ is a finite union of disjoint open sets $\{U_i\}^l_{i=1}$ such
that $\pi|_{U_i} : U_i \to U$ is a bijective isometry for every $i$.
\end{definition}

\begin{definition}[Group action of functions and sets]
    Given a function $u$ on $\mathbb{R}^n$ and
$g \in \Gamma$, we denote $(g \cdot  u)(x) = u(g\cdot x)$.
Given a subset $A\subset \mathbb{R}^n$ and $g\in \Gamma$, we denote $g\cdot A=\{g\cdot x\mid x\in A\}$.
\end{definition}

\begin{definition}[$\Gamma$-symmetric functions and sets]
    We say that a function $u$ on $\mathbb{R}^n$ is $\Gamma$-symmetric if for every $g\in \Gamma$, we have $g \cdot u = u$.
    Similarly, we say that a subset $A\subset \mathbb{R}^n$ is $\Gamma$-symmetric if for every $g\in \Gamma$, we have $g \cdot A = A$.
\end{definition}

The following lemma shows that $\Gamma$-symmetric functions on $\mathbb{R}^n$ arise exactly as compositions with $\pi$ of functions on $\mathbb{R}^n/\Gamma$.
\begin{lemma}\label{lem:uni_prop_quotient}
    If a function $u$ on $\mathbb{R}^n$ is $\Gamma$-symmetric, then there is a function $\tilde{u}$ on $\mathbb{R}^n/\Gamma$ such that $u=\tilde{u}\circ \pi$.
    Conversely, if $\tilde{u}$ is a function on $\mathbb{R}^n/\Gamma$, then $u:=\tilde{u}\circ \pi$ is $\Gamma$-symmetric.
\end{lemma}
\begin{proof}
    This follows from the universal property of quotient spaces.
\end{proof}


\begin{definition}[$\Gamma$-symmetric minimizers against $\Gamma$-symmetric competitors]
    We say a $\Gamma$-symmetric nonnegative $u\in W^{1,2}_{\text{loc}}(\mathbb{R}^n)$ is a minimizer of the one-phase problem \eqref{1-phase_Ber} against $\Gamma$-symmetric competitors if for any $R>0$ and any $\Gamma$-symmetric $v\in W^{1,2}_{\text{loc}}(\mathbb{R}^n)$ such that $u=v$ outside $B_R(0)$, we have
    \begin{equation*}
        E(u, B_R(0))\leq E(v, B_R(0)).
    \end{equation*}
\end{definition}
\par
The following key lemma allows us to compare the value of $E$ of functions on $\mathbb{R}^n/ \Gamma$ with that of their compositions with the projection map $\pi$.
\begin{lemma}\label{lem:value_E_pullback}
    Let $\tilde{u}\in W^{1,2}_{\text{loc}}(\mathbb{R}^n/\Gamma, \mathcal{H}^n)$.
    Then for any bounded measurable subset $U\subset \mathbb{R}^n/\Gamma$, we have
    \begin{equation*}
        lE(\tilde{u},U)=E(\tilde u\circ \pi, \pi^{-1}(U)),
    \end{equation*}
    where $l=|\Gamma|$.
\end{lemma}
\begin{proof}
    Following the construction in \cite[Lemma 3.4]{CF24}, let $\{A_i\}_{i\in \mathbb{N}}$ be a collection of cover charts covering $(\mathbb{R}^n \backslash \{0\})/ \Gamma$.
    Such a collection exists since every point has a neighborhood that is a cover chart.
    Define $B_1 := A_1$ and $B_{i+1} := A_{i+1}\backslash \cup^i_{j=1} A_j$.
    Then $\{B_i\}_{i\in \mathbb{N}}$ is a countable collection of disjoint measurable subsets which covers $(R^n \backslash \{0\})/ \Gamma$, and each set $B_i$ is contained in a cover chart of $(R^n\backslash\{0\})/ \Gamma$. 
    \par
    For every $i$, we write the preimage $\pi^{-1}(A_i)$ as a disjoint union $\sqcup _{j=1}^l A^j_i$ such that for every integer $1 \leq j \leq l$ the restriction
    \begin{equation*}
        \pi_{A^j_i}:A^j_i\to A_i
    \end{equation*}
is a bijective isometry.
Hence we have
\begin{equation*}
    E(\tilde{u}, B_i\cap U)=E(\tilde{u}\circ \pi, (\pi_{A^j_i})^{-1}(B_i\cap U))
\end{equation*}
for every $j = 1, \ldots,l$.
Summing over $j$ we then get that for every $i$ it holds
\begin{equation*}
    l E(\tilde{u}, B_i\cap U)=E(\tilde{u}\circ \pi, \pi^{-1}(B_i\cap U)).
\end{equation*}
Now, the collection $\{\pi^{-1}(B_i)\}_{i\in \mathbb{N}}$ is also a disjoint cover of $\mathbb{R}^n \backslash \{0\}$.
Together with the fact that $E(\tilde{u},B_\epsilon(0)/\Gamma),E(\tilde{u}\circ \pi,B_\epsilon(0)) \to 0$ as $\epsilon\to 0$, we obtain
\begin{equation*}
        lE(\tilde{u},U)=\sum_{i\in \mathbb{N}} lE(\tilde{u}, B_i\cap U)
    =\sum_{i\in \mathbb{N}} E(\tilde{u}\circ \pi, \pi^{-1}(B_i\cap U))
    =E(\tilde{u}\circ \pi, \pi^{-1}(U)).
\end{equation*}
This finishes the proof.
\end{proof}

The next lemma is the main goal of this subsection. 
It shows the correspondence between global minimizers of the one-phase problem on $\mathbb{R}^n/ \Gamma$ and $\Gamma$-symmetric minimizers against $\Gamma$-symmetric
competitors on $\mathbb{R}^n$.
\begin{lemma}\label{lem:corresp_sym_min_to_min_on_cone}
    Let $\tilde{u}$ be a global minimizer of the one-phase problem \eqref{1-phase_Ber} on $\mathbb{R}^n/ \Gamma$.
    Then $\tilde{u}\circ \pi$ is a $\Gamma$-symmetric minimizer against $\Gamma$-symmetric competitors on $\mathbb{R}^n$.
Conversely, if $u$ is a $\Gamma$-symmetric minimizer against $\Gamma$-symmetric competitors on $\mathbb{R}^n$, then the $\tilde{u}$ from Lemma \ref{lem:uni_prop_quotient} such that $u=\tilde{u}\circ \pi$ is a global minimizer of the one-phase problem on $\mathbb{R}^n/ \Gamma$.
\end{lemma}
\begin{proof}
We here prove the first statement.
    $\tilde{u}\circ \pi$ is $\Gamma$-symmetric by Lemma \ref{lem:uni_prop_quotient}. 
    For any $R > 0$ and any $\Gamma$-symmetric function $v\in W^{1,2}_{\text{loc}}(\mathbb{R}^n)$ such that $v=\tilde{u}\circ \pi$ outside $B_R(0)$, by Lemma \ref{lem:value_E_pullback} we get
     \begin{equation*}
         E(\tilde{u}\circ \pi, B_R(0))=l^{-1}E(\tilde{u}, B_R(0)/\Gamma)
         \leq l^{-1}E(\tilde{v} , B_R(0)/\Gamma)=E(v, B_R(0)),
     \end{equation*}
     where $\tilde{v}$ is the function on $\mathbb{R}^n/ \Gamma$ from Lemma \ref{lem:uni_prop_quotient} such that $v=\tilde{v}\circ \pi$. 
     We hence conclude that $\tilde{u}\circ \pi$ is a $\Gamma$-symmetric minimizer against $\Gamma$-symmetric competitors on $\mathbb{R}^n$.
     The converse statement can be proved in a similar way.
\end{proof}
We end this subsection with the following proposition.
It shows that symmetric minimizers enjoy the same regularity results as the usual minimizers.

\begin{proposition}\label{prop:sym_min_prop}
    Let $u$ be a $\Gamma$-symmetric minimizer against $\Gamma$-symmetric competitors on $\mathbb{R}^n$.
    Then $u$ is a spatially-local minimizer on $\mathbb{R}^n\backslash\{0\}$, in the sense that for any $x\in \mathbb{R}^n\backslash\{0\}$, there exists $r>0$ such that $u$ is a minimizer on $B_r(x)$.
    In particular, the regularity results in Theorem \ref{thm:Euc_one-phase_fund} and the dimension bound \eqref{eq:dim_bd_euc_1-phase_sing} hold for $u$.
\end{proposition}
\begin{proof}
    Let $\tilde{u}$ be from Lemma \ref{lem:FBP_2dim_cone} such that $u=\tilde{u}\circ \pi$.
    By Lemma \ref{lem:corresp_sym_min_to_min_on_cone}, $\tilde{u}$ is a minimizer on $\mathbb{R}^n/\Gamma$.
    Since the restricted projection map $\pi : \mathbb{R}^n \backslash \{0\} \to  (\mathbb{R}^n/\Gamma) \backslash \{0\}$ is a local isometry, $u$ is then a spatially-local minimizer  on $\mathbb{R}^n\backslash\{0\}$.
\end{proof}
Most importantly, from this proposition, $u$ solves the boundary value problem \eqref{eq:PDE} on $\mathbb{R}^n\backslash\{0\}$, and $\pt\{u>0\}\backslash\{0\}$ is smooth when $n$ is less than the critical dimension.
These enable us to use the computations in \cite{CJK04, JS15} for our symmetric minimizers.

\subsection{Necessary Conditions for Stability Under Symmetric Variations}\label{subsec:cond_stab_sym_var}

In this subsection, we first prove a necessary condition for stability under symmetric variations.
This will then be rephrased in terms of the existence of strict subsolutions to the linearized equation of \eqref{1-phase_Ber}.
Such subsolutions on $\mathbb{R}^4$ will be constructed in the next section to show the instability of any nonlinear $1$-homogeneous candidates for symmetric minimizers.
Throughout this subsection, $\nu$ denotes the outward unit normal vector of a domain, depending on context.

\begin{lemma}\label{lem:neccessary_cond_stab_sym}
    Let $n\geq 3$ be an integer.
    Let $u$ be a $\Gamma$-symmetric minimizer of the one-phase problem \eqref{1-phase_Ber} with $Q=1$ against $\Gamma$-competitors on $\mathbb{R}^n$.
    Suppose $u$ is $1$-homogeneous, and $\pt\{u>0\}$ is smooth except possibly at $0$.
    Then for any $\Gamma$-symmetric nonnegative $f\in C^\infty_0(\mathbb{R}\backslash\{0\})$, we have
    \begin{equation}\label{eq:stab_criter}
        \int_{\pt \{u>0\}}Hf^2 d\sigma\leq \int_{\{u>0\}}|\nabla f|^2 dx,
    \end{equation}
    where $H$ is the outward mean curvature of $\pt\{u>0\}$.
\end{lemma}
This lemma is a slight generalization of \cite[Lemma 1]{CJK04}, and it is proved by modifying the proof there to the symmetric setting.
Notice that by Proposition \ref{prop:sym_min_prop}, the regularity assumption on $\pt\{u>0\}$ holds when $n$ is less than the critical dimension.

\begin{proof}
    We prove the lemma by showing that \eqref{eq:stab_criter} is implied by the stability of $u$ with respect to $\Gamma$-symmetric variations.
    \par
    \textbf{Observation}: By Proposition \ref{prop:sym_min_prop}, $u$ satisfies the boundary value problem \eqref{eq:PDE}.
    This observation makes many computations in the proof of \cite[Lemma 1]{CJK04} valid in our case.
    \par
    We begin with the following construction.
    Fix a ball $B = B_R(0)$ such that $\supp f\subset B$.
Define $D = \{u>0\}\cap  B$, and let $F$ (depending on $R$) be the harmonic function in $D$ with boundary values $f$ on $\pt D$.
Then $F$ and $u$ are both nonnegative harmonic functions on $D$, which is a nontangentially accessible
domain as defined in \cite{JK82}.
By the boundary Harnack principle (\cite[Theorem 5.1]{JK82}), we have $F\leq Cu$ on a neighborhood of the origin and also in a neighborhood of $\pt B$ in $\bar{D}$, for some constant $C$ depending on $D$. 
For $\epsilon>0$, define
\begin{equation*}
    D_\epsilon=\{x\in D \mid u(x)> \epsilon F(x)\}.
\end{equation*}
Then for small enough $\epsilon$, $\pt D_\epsilon$ agrees with $\pt D$ near the origin and near $\pt B$, and it is smoothly bounded elsewhere.
Define our competitor $v_\epsilon=u-\epsilon F$ on $\bar{D_\epsilon}$ and $v_\epsilon=0$ elsewhere.
Then $v_\epsilon \geq 0$.
We also check that $F, D_\epsilon, v_\epsilon$ are $\Gamma$-symmetric since the inputs $u,f,D$ are. 
\par
Let's estimate $E(v_\epsilon, B)$.
Since $v_\epsilon = u$ on $\pt B$, integration by parts gives
\begin{equation*}
    \int_{B\cap \{v_\epsilon>0\}}|\nabla v_\epsilon|^2 dx
    =\int_{(\pt B)\cap \{u>0\}} u\pt_\nu (u-\epsilon F)d\sigma
    =\int_D |\nabla u|^2dx-\epsilon \int_{(\pt B)\cap \{u>0\}} u\pt_\nu Fd\sigma.
\end{equation*}
Next, since $F=0$ on $(\pt B)\cap \{u>0\}$, $u=0$ on $\pt \{u>0\}$, and $\pt_\nu u=-1$ on $\pt\{u>0\}\backslash\{0\}$, integration by parts once again gives
\begin{equation*}
\begin{split}
    \int_{(\pt B)\cap \{u>0\}} u\pt_\nu Fd\sigma
    &=\int_{(\pt B)\cap \{u>0\}} (u\pt_\nu F-\pt_\nu u F) d\sigma\\
    &=\int_{\pt D}(u\pt_\nu F-\pt_\nu u F) d\sigma
    -\int_{(\pt \{u>0\})\cap B} Fd\sigma.
\end{split}
\end{equation*}
The first integral of the last line vanishes since both $u,F$ are harmonic on $D$.
Combining the results above, with $F=f$ on $\pt D$, we have
\begin{equation}\label{eq:neccessary_cond_stab_sym_first_est}
    E(u,B)-E(v_\epsilon,B)=-\epsilon \int_{(\pt \{u>0\})\cap B} fd\sigma
    +\Vol(0<u<\epsilon F).
\end{equation}
To estimate the last volume term, we quote the following computational results from \cite[Lemma 1]{CJK04}:
\begin{enumerate}
    \item $\pt_\nu \pt_\nu u=-H$ on $\pt\{u>0\}$ except at $0$.\label{item:mean_curva_2nd_nor_deri_u}

    \item $H\geq 0$ on $\pt\{u>0\}$ except at $0$.

    \item (Jacobian determinant estimate)
    Let $z : U \to  \mathbb{R}^n$ defined on an open subset $U \subset \mathbb{R}^{n-1}$ be a local parametrization of $\pt\{u>0\}$ and let $\nu(s)$ be the outward unit normal of $\pt\{u>0\}$ at $z(s)$.
    In the coordinate system $x(s, t) = z(s)- t\nu(s)$, the volume element has the asymptotics
    \begin{equation*}
        \dd\vol=(1+tH+O(t^2))\dd\sigma(s)\dd t,
    \end{equation*}
    where $d\sigma(s)$ is the area element on $\pt \{u>0\}$.
\end{enumerate}
The first and the second items hold in our case by our observation in the beginning.
The third item follows from standard computations.
\par
We now derive the range for $t$ for the domain $\{0<u<\epsilon F\}$ for each fixed $s$, in the coordinate system from Item 3 above.
The equation for $\pt D_\epsilon$ can be written as
\begin{equation*}
    u(z(s)-t\nu(s))=\epsilon F(z(s)-t\nu(s)).
\end{equation*}
Taking Taylor expansions at $t=0$, we get
\begin{equation*}
    t-H\frac{t^2}{2}+O(t^3)=\epsilon (F-t\pt_\nu F+O(t^2)),
\end{equation*}
where $H,F,\pt_\nu F$ are evaluated at $z(s)$, and the error terms $O(t^2), O(t^3)$ are uniformly bounded in $s$ since the region $(\pt D)\backslash \pt D_\epsilon$ is bounded away from $\pt B$ and $0$.
In particular we see that $t=O(\epsilon)$, and hence we have
\begin{equation*}
    -H\frac{t^2}{2}+(1+\epsilon\pt_\nu F)t-\epsilon F +O(\epsilon^3)=0.
\end{equation*}
Solving $t$ in terms of $\epsilon$ implicitly, with $s$ fixed, we get our bound:
\begin{equation*}
    t=\epsilon F+\epsilon^2(F^2 \frac{H}{2}-F\pt_\nu F)+O(\epsilon^3).
\end{equation*}
With this, we can now estimate
\begin{align*}
    \Vol(0<u<\epsilon F)&=\int_{(\pt\{u>0\})\cap B}\int_0^{\epsilon F+\epsilon^2(F^2 \frac{H}{2}-F\pt_\nu F)+O(\epsilon^3)} (1+tH+O(t^2))\dd t\dd \sigma\\
    &=\int_{(\pt\{u>0\})\cap B} (\epsilon F+\epsilon^2 F^2 H-\epsilon^2 F\pt_\nu F)\dd \sigma +O(\epsilon^3).
\end{align*}
Plugging this back into our result \eqref{eq:neccessary_cond_stab_sym_first_est} above, we get
\begin{equation*}
    E(u,B)-E(v_\epsilon,B)=\epsilon^2 \int_{(\pt \{u>0\})\cap B} (F^2H-F\pt_\nu F)\dd\sigma + O(\epsilon^3).
\end{equation*}
Since $u$ is a minimizer against $\Gamma$-symmetric competitors, with $F=f$ on $(\pt \{u>0\})\cap B$ and an integration by parts, we get
\begin{equation*}
    \int_{(\pt \{u>0\})\cap B} f^2H \dd \sigma \leq \int_{(\pt \{u>0\})\cap B}F\pt_\nu F \dd\sigma
    =\int_D |\nabla F|^2 dx.
\end{equation*}
Finally, since $F$ is the harmonic function that solves the boundary value problem from $f$, we get
\begin{equation*}
    \int_D |\nabla F|^2 \dd x\leq \int_D |\nabla f|^2\dd x.
\end{equation*}
This finishes the proof.
\end{proof}

We now deduce a corollary of the lemma above, following \cite{JS15}.
It is phrased in terms of the existence of certain strict subsolutions.
To begin with, we observe that the linearized equation to the boundary value problem \eqref{eq:PDE} for $u$ is
\begin{equation}\label{eq:linearized_PDE}
    \begin{cases}
        \Delta v = 0 & \text{in $\{u>0\}$},\\
        \pt_{\nu} v=Hv&\text{on $\pt\{u>0\}\backslash\{0\}$},
    \end{cases}
\end{equation}
where $H$ is the outward mean curvature of $\pt \{u>0\}$, and we use that $H=-\pt_\nu \pt_\nu u$ (Item \ref{item:mean_curva_2nd_nor_deri_u} in the proof above).
This linearized equation can be deduced by requiring that $(u+\epsilon v)_+$ solves the boundary value problem up to an error of order $O(\epsilon^2)$ (see p.\@ 1243 in \cite{JS15}).
\par
\noindent\textbf{Convention}: Notice a different sign convention: our $\nu$ is outward while $\cite{JS15}$ used the inward unit normal of $\pt\{u>0\}$.
However, $H$ is the outward mean curvature in both cases.
\par

Let $A$ be an annulus centered at $0$.
We say that $v$ is a subsolution to the linearized equation \eqref{eq:linearized_PDE} on $\{u>0\}\cap A$ if 
\begin{equation}\label{eq:subsol_linearized_PDE}
    \begin{cases}
        \Delta v\geq 0, v\geq 0 & \text{in $\{u>0\}\cap A$},\\
        \pt_\nu v\leq Hv & \text{on $(\pt\{u>0\})\cap A$},\\
        v=0& \text{on $\{u>0\}\cap \pt A$}.
    \end{cases}
\end{equation}
In connection with our lemma \ref{lem:neccessary_cond_stab_sym} above, by integration by parts we see that
\begin{equation}\label{eq:int_by_parts_subsol}
    \begin{split}
                        \int_{\{u>0\}}|\nabla v|^2\dd x
    &=-\int_{\{u>0\}} v\Delta v\dd x
    +\int_{\pt\{u>0\}}v\pt_\nu v \dd \sigma\\
    &\leq \int_{\pt\{u>0\}}v\pt_\nu v \dd \sigma\\
    &\leq \int_{\pt\{u>0\}}v^2H \dd \sigma.
    \end{split}
\end{equation}
Therefore, we have the following:
\begin{corollary}\label{cor:strict_subsol_unstable}
    Let $u$ be as in Lemma \ref{lem:neccessary_cond_stab_sym}.
    If there is a $\Gamma$-symmetric strict subsolution $v$ to \eqref{eq:subsol_linearized_PDE}, in the sense that it makes any of the inequalities in \eqref{eq:int_by_parts_subsol} strict, then $u$ is unstable on $A$ under $\Gamma$-symmetric variations.
\end{corollary}

\par
Next, we give more specific propositions based on this corollary, which are modifications of Propositions 2.1 and 2.2 in \cite{JS15}.
Following the ideas there, we impose the homogeneity ansatz on the subsolutions.
By restricting to the unit sphere $S=S^{n-1}$, we then have the following result:
\begin{proposition}\label{prop:unstab_sym_sphere_var}
    Let $u$ be as in Lemma \ref{lem:neccessary_cond_stab_sym}.
    Suppose there is a $\Gamma$-symmetric nonnegative strict subsolution $\phi$ to the equation on the sphere
    \begin{equation*}
        \begin{cases}
            \Delta_S \phi \geq \lambda \phi & \text{in $D_S:=\{u>0\}\cap S$},\\
            \pt_\nu \phi \leq H\phi & \text{on $\pt D_S$},
        \end{cases}
    \end{equation*}
    with the constant $\lambda \geq (n-2)^2/4$; here $\Delta_S$ is the Laplacian on $S$.
    Then $u$ is unstable under $\Gamma$-symmetric variations on some sufficiently large annulus.
\end{proposition}
\begin{proof}
    Our proof closely follows that of \cite[Peoposition 2.1]{JS15}.
    Since $\phi$ is a strict subsolution, a similar integration by parts as in \eqref{eq:int_by_parts_subsol} above gives
    \begin{equation*}
        \int_{D_S} |\nabla \phi|^2\dd x -\int_{\pt D_S}H\phi^2\dd \sigma
        <-\lambda \int_{D_S}\phi^2 \dd x.
    \end{equation*}
    Hence, we have $\lambda<\Lambda$, where  
    \begin{equation}\label{eq:sphere_min_prob}
        -\Lambda=\inf_{\psi:\text{$\Gamma$-sym.}} \frac{\int_{D_S} |\nabla \psi|^2\dd x -\int_{\pt D_S}H\psi^2\dd \sigma}
        {\int_{D_S}\psi^2 \dd x}.
    \end{equation}
By classical calculus of variations, the minimizer $\bar{\psi}$ to the above minimizing problem \eqref{eq:sphere_min_prob} exists and satisfies $\bar{\psi} \in C^\infty (\bar{D_S})$, $\bar{\psi}>0$ in $D_S$ and 
\begin{equation*}
    \begin{cases}
        \Delta_S \bar{\psi}=\Lambda \bar{\psi} &\text{in $D_S$},\\
        \pt_\nu \bar{\psi}=H\bar{\psi} &\text{on $\pt D_S$}.
    \end{cases}
\end{equation*}
We extend $\bar{\psi}$ to the whole $\{u>0\}\subset \mathbb{R}^n$ by $0$-homogeneity and define $v=f(r)\bar{\psi}$, where $r=|x|$ and $f$ is to be determined.
Then we have
\begin{equation*}
    \Delta v= (f''+(n-1)f'/r+\Lambda f/r^2)\bar{\psi}.
\end{equation*}
Now, we choose a constant $\beta$ such that 
\begin{equation*}
    \Lambda>\beta>\lambda\geq (n-2)^2/4,
\end{equation*}
and let $f$ solves the ODE
\begin{equation*}
    f''+(n-1)f'/r+\beta f/r^2=0.
\end{equation*}
It can be checked that the values of $f$ oscillate around $0$.
Hence, we can choose an open annulus $A\subset \mathbb{R}^n$ centered at $0$ such that $f>0$ in $A$ and $f=0$ on $\pt A$.
Then we have
\begin{itemize}[topsep=0pt, ]
    \item $v>0$ on $\{u>0\}\cap A$,
    \item $\Delta v=(\Lambda -\beta)f \bar{\psi}/r^2>0\quad \text{in $\{u>0\}\cap A$}$,
    \item $\pt_\nu v=Hv$ on $(\pt\{u>0\})\cap A$ since $f$ is radial,
    \item $v=0$ on $\{u>0\}\cap \pt A$ since $f=0$ on $\pt A$,
    \item $v$ is $\Gamma$-symmetric since $\bar{\psi}$ is $\Gamma$-symmetric.
\end{itemize}
    Therefore, $v$ serves as a $\Gamma$-symmetric strict subsolution in Corollary \ref{cor:strict_subsol_unstable}, and hence $u$ is unstable on $A$.
\end{proof}
\par
We also have the following homogeneous version:
\begin{proposition}\label{prop:unstab_sym_homoge_var}
    Let $u$ be as in Lemma \ref{lem:neccessary_cond_stab_sym}.
    Suppose there is a $\Gamma$-symmetric, homogeneous of degree $-\mu$, nonnegative strict subsolution $\bar{v}$ to the equation
    \begin{equation*}
        \begin{cases}
            \Delta \bar{v} \geq \gamma \bar{v}/|x|^2 & \text{in $\{u>0\}$},\\
            \pt_\nu \bar{v} \leq H\bar{v} & \text{on $\pt \{u>0\}\backslash\{0\}$},
        \end{cases}
    \end{equation*}
    with the constant $\gamma$ satisfying
    \begin{equation}\label{eq:cond_gamma}
        \gamma \geq (\frac{n}{2}-1-\mu)^2.
    \end{equation}
    Then $u$ is unstable under $\Gamma$-symmetric variations.
\end{proposition}
\begin{proof}
    If we restrict $\bar{v}$ to the unit sphere $S=S^{n-1}$, then it satisfies
    \begin{equation*}
        \Delta_S \bar{v}\geq (\gamma+\mu(n-2-2\mu))\bar{v}\quad \text{on $\{u>0\}\cap S$}.
    \end{equation*}
    In connection with the previous Proposition \ref{prop:unstab_sym_sphere_var}, we see that the inequality 
    \begin{equation*}
        \gamma+\mu(n-2-2\mu)\geq (n-2)^2/4
    \end{equation*}
    is the same as \eqref{eq:cond_gamma}.
    The conclusion now follows from the previous proposition.
\end{proof}

In the next section, we will construct such a strict subsolution on $\mathbb{R}^4$ as in the last proposition above.
This will show the instability of any nonlinear $1$-homogeneous candidates for symmetric minimizers.

\section{Proof of the Main Results}\label{sec:proof_main}

In this section, we will first prove the rigidity Theorem \ref{thm:rigidity}, and then use it to prove Theorem \ref{thm:dim_bound}.
For the reader's convenience, we restate the first theorem here with slightly different notations:

\begin{theorem}[Theorem \ref{thm:rigidity} restated]
Let $\Gamma$ be a discrete group of isometries of $S^3$ acting freely.
    Let $C(S^3/\Gamma)$ be the metric cone over $S^3/\Gamma$, and let $p$ be its tip.
    If $\tilde{u}$ is global minimizer of the one-phase problem \eqref{1-phase_Ber} on $C(S^3/\Gamma)$ with constant $Q$ such that $p \in \pt \{\tilde{u}>0\}$, then $\Gamma$ is trivial, and $\tilde{u}$ is a linear solution on $\mathbb{R}^4\cong C(S^3)$ as in \eqref{eq:linear_sol}.
\end{theorem}
Here is our plan for the proof.
By passing to a blowup along with the tools we have developed in subsection \ref{subsec:sym_obj}, it suffices to consider $1$-homogeneous $\Gamma$-symmetric minimizers $u$ against $\Gamma$-symmetric competitors on $\mathbb{R}^4$.
Next, we prove the linearity of $u$ by contradiction: Assuming $u$ is nonlinear, we construct a subsolution as in Proposition \ref{prop:unstab_sym_homoge_var}, resulting in the instability of $u$.
The construction of such subsolutions closely follows \cite[Section 4]{JS15} and uses $u$ as input.
Finally, we use the linearity of $u$ to deduce that $\Gamma$ is trivial, and hence $\tilde{u}$ is linear.

\begin{proof}[Proof of Theorem \ref{thm:rigidity}]
Without loss of generality, we assume $Q=1$.
We also use the identification $C(S^3/\Gamma)\cong \mathbb{R}^4/\Gamma$ and $p=0$.
Let $\pi:\mathbb{R}^4\to \mathbb{R}^4/\Gamma$ be the quotient map.
Let $\tilde{u}_0$ be a blowup of $\tilde{u}$ at $0$.
By Lemma \ref{lem:RCD_blowup_prop} and Theorem \ref{thm:rcd_cone_weiss}, $\tilde{u}_0$ is $1$-homogeneous global minimizer on $\mathbb{R}/\Gamma$.
By Lemma \ref{lem:corresp_sym_min_to_min_on_cone}, there exists a $\Gamma$-symmetric minimizer $u$ on $\mathbb{R}^4$ to the one-phase problem against $\Gamma$-symmetric competitors, such that $0\in \pt\{u>0\}$ and $u=\tilde{u}_0\circ\pi$.
Clearly $u$ is also $1$-homogeneous. 
By Proposition \ref{prop:sym_min_prop}, $u$ satisfies the boundary value problem \eqref{eq:PDE} and hence the computations in \cite{JS15} are valid in our case.
\par
Assuming $u$ is not linear, we construct a strict subsolution as in Proposition \ref{prop:unstab_sym_homoge_var}.
Define the following function
\begin{equation}\label{eq:dim_4_constr_w}
    w=\sqrt{\sum_{\lambda_j>0}\lambda^2_j + 4\sum_{\lambda_j<0}\lambda^2_j},
\end{equation}
where $\lambda_j$ are the eigenvalues of $D^2 u$.
Notice that since $u$ is $\Gamma$-symmetric, so is $w$.
Furthermore, since $|\nabla u|^2/2$ is subharmonic on $\{u>0\}$ and homogeneous of degree $0$, its maximum happens on $\pt\{u>0\}$.
Then since $u$ is nonlinear, by the Hopf lemma and Item \ref{item:mean_curva_2nd_nor_deri_u} in Lemma \ref{lem:neccessary_cond_stab_sym} we know that
\begin{equation}\label{eq:u_concv_bdy_normal}
    0<\pt_{\nu} (|\nabla u|^2/2)=-\pt_\nu\pt_\nu u=H\quad \text{ on $\pt\{u>0\}\backslash\{0\}$},
\end{equation}
where $\nu$ is the outer unit normal (in \cite{JS15} the $\nu$ there is \emph{inner}), and $H$ is the outward mean curvature.
On the other hand, one checks that $\nu$ is an eigenvector of $D^2u$ on $\pt\{u>0\}\backslash\{0\}$(See \cite[Subsection 2.1]{JS15}).
This implies there is some $\lambda_j<0$, and hence $w>0$ on $\pt \{u>0\}\backslash\{0\}$.
\par
Next, we quote the following results:
\begin{lemma}[{\cite[Theorem 4.1]{JS15}}]\label{lem:ineq_w}
    For $w$ defined in \eqref{eq:dim_4_constr_w}, we have the interior inequality
    \begin{equation}\label{eq:w_interior_ineq}
        w\Delta w\geq \frac{2}{3}|\nabla w|^2+\frac{4}{3}\frac{w^2}{|x|^2}\quad \text{in $\{u>0\}$}
    \end{equation}
    in the viscosity sense.
    On the free boundary, we have
    \begin{equation}\label{eq:w_bdy_ineq}
        \frac{1}{H}\frac{\pt_\nu w}{w}\leq 3 \quad \text{on $\pt \{u>0\}\backslash\{0\}$}.
    \end{equation}
\end{lemma}
These inequalities can be proved by choosing a coordinate system at a point that diagonalizes $D^2u$ and using properties of convex, symmetric, homogeneous functions of the eigenvalues of $D^2 u$ alongside the equations \eqref{eq:PDE}.
See \cite[Section 4]{JS15} for details.
\par
We now define
\begin{equation*}
    \bar{v}=w^\frac{1}{3},
\end{equation*}
where $w$ is in \eqref{eq:dim_4_constr_w}.
Then $\bar{v}$ is $\Gamma$-symmetric and homogeneous of degree $\frac{-1}{3}$.
Using Lemma \ref{lem:ineq_w} above and comparing with Proposition \ref{prop:unstab_sym_homoge_var}, we have
\begin{itemize}
    \item On $\{u>0\}$,
\begin{equation}\label{eq:v_interior_ineq}
    \begin{split}
        \Delta \bar{v}&=w^\frac{-5}{3}(\frac{-2}{9}|\nabla w|^2+\frac{1}{3}w\Delta w)\\
        &\geq w^\frac{-5}{3}\left(\frac{-2}{9}|\nabla w|^2+\frac{1}{3}(\frac{2}{3}|\nabla w|^2+\frac{4}{3}\frac{w^2}{|x|^2})\right)
        =\frac{4}{9}\frac{\bar{v}}{|x|^2},
    \end{split}
\end{equation}
    and $\frac{4}{9}\geq (2-1-\frac{1}{3})^2$.

    \item On $\pt \{u>0\}\backslash\{0\}$,
    \begin{equation}\label{eq:v_bdy_ineq}
        \pt_\nu \bar{v}=\frac{1}{3}w^\frac{-2}{3}\pt_\nu w
        \leq H\bar{v}.
    \end{equation}
\end{itemize}
Thus $\bar{v}$ is a subsolution as described in Proposition \ref{prop:unstab_sym_homoge_var}.
To show that $\bar{v}$ is a strict subsolution, we fix $x_0\in \pt\{u>0\}\backslash\{0\}$ and consider the following two cases.
\begin{enumerate}
    \item If the strict inequality in \eqref{eq:v_bdy_ineq} holds at $x_0$, then by continuity we are done .

    \item Suppose the equality in \eqref{eq:v_bdy_ineq} holds at $x_0$.
    Then the equality in \eqref{eq:w_bdy_ineq} also holds.
    We choose a coordinate system near $x_0$ such that $e_4=-\nu$ at $x_0$, $e_1$ is radial, and $D^2 u(x_0)$ is diagonal with entries $\lambda_i$.
    A priori, $\lambda_1=0$ since $u$ is $1$-homogeneous, and $\lambda_4=-H<0$ as shown in \eqref{eq:u_concv_bdy_normal}.
    Using \cite[Line (4.9)]{JS15}, $\lambda_2, \lambda_3$ have the opposite signs; if we name them so that $\lambda_2>0$ and $\lambda_3<0$, then we have
    \begin{equation*}
        \lambda_2=2H,\, \lambda_3=-H=\lambda_4.
    \end{equation*}
    Plugging these back into the computation for \eqref{eq:w_interior_ineq}, we then have, as shown in the remark of \cite[Theorem 4.1]{JS15},
\begin{equation*}
        w\Delta w\geq \frac{2}{3}|\nabla w|^2+\frac{4}{3}\frac{w^2}{|x|^2}+\epsilon (\pt_{\nu}w)^2\quad \text{near $x_0$ in $\{u>0\}$},
    \end{equation*}
    for some $\epsilon>0$ depending on $H$. 
    This then results in an $\epsilon$-strict inequality for $\bar{v}$ in \eqref{eq:v_interior_ineq}.
\end{enumerate}
Since we have, assuming $u$ is nonlinear, successfully constructed a strict subsolution $\bar{v}$ in Proposition \ref{prop:unstab_sym_homoge_var}, we conclude that $u$ is unstable under $\Gamma$-symmetric variations.
This is a contradiction since $u$ is a minimizer against $\Gamma$-symmetric competitors.
Therefore, $u$ has to be a linear solution.
    \par
    Let us show that $\Gamma = \{\id_{S^3}\}$, following \cite[p.\@ 381]{CF24}.
    Since $\{u>0\}$ is a half space, $\pt \{u>0\}\cap S^3=S^2$.
    Since $u$ is $\Gamma$-symmetric, $\pt \{u>0\}\cap S^3$ is sent to itself by all elements of $\Gamma$.
    We claim that the two poles of $S^3$ with respect to  $\pt \{u>0\}\cap S^3$ (that is, the two points on $S^3$ at maximal distance from  $\pt \{u>0\}\cap S^3$) are swapped by every element of $\Gamma$ that is not the identity.
    \par
    Indeed, the poles cannot be fixed as the action of $\Gamma$ is free.
    Furthermore, the distance between each pole and  $\pt \{u>0\}\cap S^3$ must be preserved since all elements of $\Gamma$ are isometries.
    In particular, all the elements of $\Gamma$ which are not the identity swap the poles.
    Furthermore, since $\Gamma$ acts on $\mathbb{R}^4$ isometrically, for every $g \in \Gamma, g \neq \id_{S^3}$, there is a neighborhood $U\subset \mathbb{R}^4$ of one of the poles such that $U \subset \{u>0\}$ and $g \cdot U \subset \mathbb{R}^4\backslash \{u>0\}$.
    Consequently, any such element of $\Gamma$ cannot map the half space $\{u>0\}$ to itself.
    As $\{u>0\}$ is $\Gamma$-symmetric, we conclude that $\Gamma$ is trivial.
    \par
    Now that $\Gamma$ is trivial, by Corollary \ref{cor:global_min_linear} $\tilde{u}$ is a linear solution.
\end{proof}

As we mentioned in the introduction, it is conjecturally possible to prove Theorem \ref{thm:rigidity} for higher-dimensional spheres.
However, it fails for $S^{n-1}$ for $n\geq 7$ since we have the following example:
\begin{example}
    Let $u$ be the singular $1$-homogeneous global minimizer with $Q=1$ on $\mathbb{R}^7\cong C(S^6)$ from \cite{DJ09}, with $0\in \pt\{u>0\}$.
For any $n = 7+k$ with $k \in \mathbb{N}_0$, define $u'(x,y)$ on $\mathbb{R}^7\times \mathbb{R}^k$ by
\begin{equation*}
    u'(x,y)=u(x).
\end{equation*}
Then by Remark \ref{rmk:sec_min_min_corresp}, $u'$ is a $1$-homogeneous global minimizer on $\mathbb{R}^n\cong C(S^{n-1})$ and $0\in \{u'>0\}$, but $u'$ is not a linear solution.
\end{example}

We now prove Theorem \ref{thm:dim_bound} using Theorem \ref{thm:rigidity}.
For the reader's convenience, we restate the theorem here (also recall Definitions \ref{def:RCD_FBP_reg_sing} and \ref{def:strata_both_asp} about singular sets and strata):
\begin{theorem}[Theorem \ref{thm:dim_bound} restated]
    Let $X$ be an $n$-dimensional noncollapsed limit space with two-sided bounds on the Ricci curvature, as defined by \eqref{intro:2sided_Ric_lim}.
    Let $u$ be a minimizer of the one-phase problem \eqref{1-phase_Ber} on some bounded domain $\Omega\subset X$.
    Then we have $\mathcal{S}(u)=\mathcal{S}(u)^{n-5}$.
    In particular, we have
    \begin{equation*}
        \dim_{H}\mathcal{S}(u)\leq n-5.
    \end{equation*}
Furthermore, this dimension bound is sharp.
\end{theorem}
\begin{proof}[Proof of Theorem \ref{thm:dim_bound}]
We show $\mathcal{S}(u)\backslash \mathcal{S}(u)^{n-5}=\emptyset$ by contradiction.
Suppose $x \in \mathcal{S}(u)\backslash \mathcal{S}(u)^{n-5}$.
Then by definition, there exists a pointed metric cone $(C(Z), p)$ such that
\begin{equation*}
    (\mathbb{R}^{n-4}\times C(Z), u_x, (0^{n-4},p)) \in \Tan(X,u, x),
\end{equation*}
where $u_x$ is a $\mathbb{R}^{n-4}$-invariant global minimizer on $\mathbb{R}^{n-4}\times C(Z)$ with $(0^{n-4}, p)\in \pt\{u_x>0\}$.
By Lemma \ref{lem:restrict_min_is_min}, $\hat{u}:=u_x|_{\{0^{n-4}\}\times C(Z)}$ is a global minimizer on $C(Z)$ with $p\in \pt\{\hat{u}>0\}$.
On the other hand, by Theorem \ref{thm:twosided_ric_n-4_rig}, we have $C(Z)\cong C(S^3/ \Gamma)$, where $\Gamma$ is a discrete group of isometries on $S^3$ acting freely.
Hence by Theorem \ref{thm:rigidity}, $\Gamma$ is trivial, $\hat{u}$ is a linear solution, and thus so is $u_x$.
This contradicts that $x\in \mathcal{S}(u)$.
\par
By \eqref{eq:dual_sing_equal_n-2_strata} and the inclusion of strata in \eqref{eq:dual_strata_inclu}, we conclude $\mathcal{S}(u)=\mathcal{S}(u)^{n-5}$, and hence $\dim_H (\mathcal{S}(u))\leq n-5$ by Lemma \ref{lem:dim_bd_strata_two_asp}.
The sharpness of this dimension bound follows from Example \ref{ex:dim_bound_sharp} below.
\end{proof}

\begin{example}\label{ex:dim_bound_sharp}
We show that there exists a noncollapsed limit of manifolds with two-sided bounds on the Ricci curvature with a global minimizer $u$ on it such that $\mathcal{S}(u)^{n-5}$ is nonempty.
\par
   The metric cone $C(\mathbb{RP}^3)$, arising as the blow-down of the Eguchi-Hanson manifold (see \cite[Example 2.15]{CN15}; the Eguchi-Hanson metric was originally defined in \cite{Cal79, EH79}), is a noncollapsed limit of manifolds with two-sided bounds on the Ricci curvature (Ricci-flat, actually).
   Clearly, it is singular at the tip $p$.
   It can be checked that $u(x,t)$ defined on $C(\mathbb{RP}^3)\times \mathbb{R}$ by
   \begin{equation*}
       u(x,t)=t_+
   \end{equation*}
   is a global minimizer of the one-phase problem with $Q=1$, and $p\times \{0\}$ belongs to $\mathcal{S}(u)^{n-5}$.
\end{example}

As mentioned in Remark \ref{rmk:FB_reg_in_R(X)}, the regularity of the free boundary can be improved if we consider its restriction to the regular set $\mathcal{R}(X)$.
The following proposition is similar to \cite[Proposition 3.12]{CF24}, and its proof follows by transforming the problem to the Euclidean space via \cite{And90,CC96}, and applying \cite{FS07} about free boundary problems on $\mathbb{R}^n$ with variable coefficients.
\begin{proposition}\label{prop:FB_reg_in_R(X)}
    Let $X$ be a noncollapsed limit of manifolds with two-sided bounds on the Ricci curvature of dimension $n$.
    Let $u$ be a minimizer of the one-phase problem on some domain $\Omega\subset X$.
    Then $\pt\{u>0\}\cap \Omega\cap\mathcal{R}(X)$ is a $C^{1,\gamma}$ hypersurface for some $\gamma$, outside a closed set of Hausdorff dimension at most $n - k^*$, where $k^*$ is the critical dimension.
\end{proposition}
\begin{proof}
    By Theorem \ref{thm:twosided_ricci_R(X)_holder}, $\mathcal{R}(X)$ is a $C^{1,\gamma}$ open Riemannian manifold.
    Hence by the same dimension reduction argument behind \eqref{eq:dim_bd_euc_1-phase_sing}, we have
    \begin{equation*}
        \dim_H(\mathcal{S}(u)\cap \mathcal{R}(X))\leq n-k^*.
    \end{equation*}
    Let's fix a point $x_0 \in \mathcal{R}(u)=\mathcal{R}(u)\cap\mathcal{R}(X)$ and show that $\pt\{u>0\}$ is a $C^{1,\gamma}$ hypersurface in a neighborhood of $x_0$.
    \par
    Fix some small $\epsilon>0$.
    By \cite{And90, CC96}, we can identify a neighborhood of $x_0$ with the Euclidean ball $B_2(0^n) \subset \mathbb{R}^n$ equipped with a Riemannian metric $g$ such that
\begin{equation*}
    \|g_{ij}-\delta_{ij}\|_{C^1(B_2(0))}<\epsilon.
\end{equation*}
With this identification, we can view $u$ as a minimizer of the one-phase problem on $B_2(0^n)$ and $0^n=x_0\in \mathcal{R}(u)$.
More precisely, $u$ is a minimizer of the functional with variable coefficients:
\begin{equation*}
    \int_{B_2(0)} g^{ij}(x)\pt_{x_i}u(x)\pt_{x_j}u(x) +Q(x)1_{u>0}(x)\dd x,
\end{equation*}
with its boundary data on $\pt B_2(0)$.
We can now apply \cite{FS07} to conclude that $\pt\{u>0\}\cap B_1(0)$ is a $C^{1,\gamma}$ graph for some $\gamma$.
\end{proof}

\appendix
\section{Appendix: Dimension Bound on Singular Strata}\label{apndx:dim_bound_strata}

In this appendix, we prove Lemma \ref{lem:dim_bd_strata_two_asp} by adapting the modified dimension-reduction arguments in \cite[Theorem 4.5]{FMS25}.
A key ingredient is the Weiss-type monotonicity formula for minimizers on $\mathsf{RCD}$ cones, which we quote here:
\begin{theorem}[{\cite[Lemma A.3]{CZZ22}}]\label{thm:rcd_cone_weiss}
    Let $u$ be a global minimizer of the one-phase problem \eqref{1-phase_Ber} on a noncollapsed $\mathsf{RCD}(0,n)$ cone $C(Z)\cong ([0,\infty)\times Z)/(\{0\}\times Z)$, with $n\geq 2$, $Q\equiv 1$, and $p$ the cone tip.
   For $r>0$, define
    \begin{equation*}
        W(r):=r^{-n}\int_{B_r(p)}(|\nabla u|^2 +1_{u>0})\dd {\mathcal{H}^n} 
        -r^{-n-1}\int_{C(Z)}u^2 \dd{|D1_{B_r(p)}|},
    \end{equation*}
    where $|D1_{B_r(p)}|$ is the total variation of $1_{B_r(p)}$.
    Then $W$ is locally Lipschitz and increasing.
    At those $r$ where $W$ is differentiable, we have
    \begin{equation*}
        W'(r)\geq r^{-1}\int_{Z} \left(\nabla u(r,z)\cdot \nabla d_{C(Z)}(p, \cdot)-\frac{u(r,z)}{r} \right)^2\dd{\mathcal{H}^{n-1}(z)}.
    \end{equation*}
    In particular, $W$ is constant if and only if $u$ is $1$-homogeneous with respect to $p$.
\end{theorem}

\begin{proof}[Proof of Lemma \ref{lem:dim_bd_strata_two_asp}]
    The proof is divided into four steps.
    In the first step, we set up a contradiction argument and reduce to the case of global minimizers on $\mathsf{RCD}$ cones.
    In the second step, we make a further reduction to the case where the minimizer is $1$-homogeneous, using Theorem \ref{thm:rcd_cone_weiss}. 
    In step three, via additional blow-up arguments, we gain a splitting for both the ambient space and the minimizer, resulting in a dimension reduction.
    Finally in step four, the contradiction argument is completed by induction on the ordinal of the stratum.
    A key subtlety, in contrast to classical situations, is that the monotonicity formula here only holds for minimizers centered at the tip of metric measure cones, necessitating additional blow-ups.
    For simplicity, we assume that $Q\equiv 1$ in the one-phase problem.
    \par
    \textbf{Step 1.} Suppose that the statement does not hold for some $k\in \mathbb{N}_0$.
    Then there exists $k'\in \mathbb{R}$ such that $k'>k$ and
    \begin{equation}\label{eq:contradic_pos_k'_Haus}
        \mathcal{H}^{k'}(\mathcal{S}^k(u))>0.
    \end{equation}
    For $\epsilon>0$, consider the quantitative $(k,\epsilon)$-stratum
    \begin{equation}
        \mathcal{S}^k_\epsilon(u):=\{x\in \pt\{u>0\}\cap \Omega \mid \text{for all $r\in (0,1)$, $B_r(x)$ is not $(k+1,\epsilon)$-symmetric}\},
    \end{equation}
    where we say that $B_r(x)$ is $(k,\epsilon)$-symmetric if     \begin{enumerate}[topsep=0pt, ]
        \item There is a noncollapsed $\mathsf{RCD}(0,n)$ cone $(Y,p)\cong(\mathbb{R}^k\times C(Z),(0^k,z_*))$ such that  
        \begin{equation}
            d_{GH}(B_r(x), B^Y_r(p)) < \epsilon r,
        \end{equation}
where $p$ is the cone tip;
        \item There exist $\tilde{u}\in W^{1,2}_{\text{loc}}(Y)$ that is $\mathbb{R}^k$-invariant (under the identification above) and a Borel $2\epsilon r$-pGH equivalence $F:B^Y_r(p)\to B_r(x)$, such that
        \begin{equation}
            r^{-2-n}\int_{B^Y_r(p)} |u\circ F-\tilde{u}|^2 \dd{\mathcal{H}^n}< \epsilon.
        \end{equation}
    \end{enumerate}
    One can check that $\mathcal{S}^k_{\epsilon_1}(u)\subset \mathcal{S}^k_{\epsilon_2}(u)$ for $\epsilon_1\geq \epsilon_2>0$, and that $\mathcal{S}^k(u)=\cup_{\epsilon>0}\mathcal{S}^k_\epsilon(u)$ (see \cite{BW26}).
    Then \eqref{eq:contradic_pos_k'_Haus} implies that
    \begin{equation*}
        \mathcal{H}^{k'}(\mathcal{S}^k_{\bar{\epsilon}}(u))>0
    \end{equation*}
    for some $\bar{\epsilon}>0$.
    By \cite[2.10.19]{federer_geometric_1996}, there exists $x\in\mathcal{S}^k_{\bar{\epsilon}}(u)$ such that
    \begin{equation*}
        \limsup_{r\to 0} \frac{\mathcal{H}^{k'}_\infty(\mathcal{S}^k_{\bar{\epsilon}}(u)\cap B_r(x))}{r^{k'}}\geq 2^{-k'},
    \end{equation*}
where $\mathcal{H}^{k'}_\infty$ is the $k'$-dimensional $\infty$-Hausdorff premeasure.
Now, we choose a sequence $r_j\to 0$ such that
\begin{enumerate}[topsep=0pt, ]
    \item $(X, d/r_j, u_j:=u/r_j,x)$ converge to $(Y,d_Y,u_x, y)\in \Tan(X,u,x)$ in the sense of definition \ref{def:tang_func}. By Remark \ref{rmk:rcd_tang_cone} and Lemma \ref{lem:RCD_blowup_prop}, $u_x$ is a global minimizer on the $\mathsf{RCD}(0,n)$ cone $Y\cong C(Z)$.

    \item $\mathcal{S}^k_{\bar{\epsilon}}(u_j)\cap B^j_1(x)\subset (X,d/r_j)$ GH-converge to some compact set $A\subset B^{C(Z)}_1(y)$, via Blaschke’s theorem (cf.\@ \cite[7.3.8]{BBI01}).
    
    \item $\limsup_{j\to \infty} \frac{\mathcal{H}^{k'}_\infty(\mathcal{S}^k_{\bar{\epsilon}}(u)\cap B_{r_j}(x))}{r_j^{k'}}\geq 2^{-k'}$.
\end{enumerate}
It is straightforward to check that $A \subset \mathcal{S}^k_{\bar{\epsilon}}(u_x)$.
Therefore, we obtain
\begin{equation*}
\begin{split}
        \mathcal{H}^{k'}_\infty(\mathcal{S}^k_{\bar{\epsilon}}(u_x)\cap B^{C(Z)}_1(y))
    &\geq \mathcal{H}^{k'}_\infty(A)
    \geq \limsup_{j\to \infty} \mathcal{H}^{k'}_\infty(\mathcal{S}^k_{\bar{\epsilon}}(u_j)\cap B^j_1(x))\\
    &=\limsup_{j\to \infty} \frac{\mathcal{H}^{k'}_\infty(\mathcal{S}^k_{\bar{\epsilon}}(u)\cap B_{r_j}(x))}{r_j^{k'}}\geq 2^{-k'},
\end{split}
\end{equation*}
    where we use the classical upper semicontinuity of the Hausdorff premeasure with respect to Hausdorff convergence (in some convergence-realizing space).
    Hence we have
    \begin{equation}\label{eq:k'_Hausdorff_measure_1st_blowup}
        \mathcal{H}^{k'}(\mathcal{S}^k_{\bar{\epsilon}}(u_x)\cap B^{C(Z)}_1(y))>0.
    \end{equation}
    \par
\textbf{Step 2.}
In this step, by another blow-up and Theorem \ref{thm:rcd_cone_weiss}, we further reduce to the case of $1$-homogeneous global minimizers.
For the sake of clarity, we define a \emph{vertex} of $C(Z)$ as a point $p\in C(Z)$ such that $C(Z)$ isometrically admits a metric cone structure with $p$ as the tip.
For example, $\mathbb{R}^n$ has all its points as vertices.
We remark that the set of vertices of $C(Z)$ is isometric to $\mathbb{R}^l$ for some $0\leq l\leq n$.
\par
We claim that there is a point $O\in \mathcal{S}^k_{\bar{\epsilon}}(u_x)$ such that $O$ is a vertex of $C(Z)$ and the following density estimate holds:
    \begin{equation*}
        \limsup_{r\to 0} \frac{\mathcal{H}^{k'}_\infty(\mathcal{S}^k_{\bar{\epsilon}}(u_x)\cap B_r(O))}{r^{k'}}\geq 2^{-k'}.
    \end{equation*}
Indeed, if the claim does not hold, then by \eqref{eq:k'_Hausdorff_measure_1st_blowup} there are points of (upper) density for $\mathcal{H}^{k'}_\infty\llcorner \mathcal{S}^k_{\bar{\epsilon}}(u_x)$ and none of them is a vertex of $C(Z)$.
Hence, we can repeat the argument in step 1, blowing up at such a non-vertex density point in the ambient cone.
In this way, the dimension of the set of vertices of the ambient space increases at least by one.
The procedure can be iterated until one of the following two possibilities occurs: Either the ambient space is isometric to $\mathbb{R}^n$, in which case \eqref{eq:k'_Hausdorff_measure_1st_blowup} contradicts the classical regularity theory (as in \eqref{eq:dim_bd_euc_1-phase_sing}), or there is a density point (of some global minimizer) that is also a vertex.
\par
Let $O$ denote any such density vertex of $C(Z)$.
We perform another blow-up at $O$ and obtain a global minimizer $v$ on $(C(Z),O)$.
By Theorem \ref{thm:rcd_cone_weiss}, $v$ is $1$-homogeneous.
In addition, by repeating the arguments in step 1, taking into account that $O$ was chosen to be a density point, it holds
    \begin{equation}\label{eq:k'_Hausdorff_measure_2nd_blowup}
        \mathcal{H}^{k'}(\mathcal{S}^k_{\bar{\epsilon}}(v)\cap B^{C(Z)}_1(O))>0.
    \end{equation}
\par
\textbf{Step 3.} The goal of this step is to gain a splitting for both the ambient space and the minimizer, by considering a blowup of $v$ at a density point for $\mathcal{H}^{k'}_\infty\llcorner \mathcal{S}^k_{\bar{\epsilon}}(v)$ that is not the tip $O$.
Our setup is that $v$ is a $1$-homogeneous global minimizer on the metric cone $(C(Z),O)$ with the tip $O$, and we have $\mathcal{H}^{k'}(\mathcal{S}^k_{\bar{\epsilon}}(v)\cap B^{C(Z)}_1(O))>0$.
By the very same arguments as in step 1, there exist a point $P\in \mathcal{S}^k_{\bar{\epsilon}}(v)\backslash\{O\}$ and a sequence $r_j\to 0$ such that $(C(Z), d_j:=d_{C(Z)}/r_j, v_j:=v/r_j,P)$ converge to $(C(Z'),d_{C(Z')},w, O')$, where $w$ is a global minimizer on the $\mathsf{RCD}(0,n)$ cone $C(Z')$; moreover, we have
    \begin{equation}\label{eq:k'_Hausdorff_measure_3rd_blowup}
        \mathcal{H}^{k'}(\mathcal{S}^k_{\bar{\epsilon}}(w)\cap B^{C(Z')}_1(O'))>0.
    \end{equation}
\par
Our goal now is to show that $C(Z')$ splits off a line isometrically, and $w$ is invariant along this line.
Let $C_j$ denote the rescaled pointed space  $(C(Z), d_j, P)$.
Consider the sequence of functions $f_j:C_j\to \mathbb{R}$ defined by
\begin{equation}
    f_j(\xi)=r_j\left(d_j^2(O,\xi)-d_j^2(O,P)\right)
\end{equation}
By \cite[Lemma 4.9]{FMS25}, the functions $f_j$ converge to some splitting function $g:C(Z')\to \mathbb{R}$ in $H^{1,2}_{\text{loc}}$; see \cite{AH17} for the relevant background.
Since $v$ is $1$-homogeneous, we have
\begin{equation*}
    \frac{1}{2}\nabla^{d_{C(Z)}} v\cdot \nabla^{d_{C(Z)}} d^2_{C(Z)}(O, \cdot)=v
\end{equation*}
$\mathcal{H}^n$-a.e.\@ on $C(Z)$.
Using the definition of $f_j$ and $v_j$, we then get
\begin{equation*}
        \frac{1}{2}\nabla^{d_j} v_j\cdot \nabla^{d_j}  f_j=r_jv_j
\end{equation*}
$\mathcal{H}^n$-a.e.\@ on $C_j$.
By \cite[Theorem 7.1]{CZZ22}, $v_j\to w$ strongly in $H^{1,2}_{\text{loc}}$.
Hence by \cite[Theorem 5.7]{AH17}, $\nabla^{d_j} v_j\cdot \nabla^{d_j}f_j$ converges to $\nabla^{d_{C(Z')}}w \cdot \nabla^{d_{C(Z')}}g$ in $L^1_\text{loc}$.
Together with $r_j\to 0$, we then have
\begin{equation*}
    \nabla^{d_{C(Z')}}w \cdot \nabla^{d_{C(Z')}}g=0
\end{equation*}
$\mathcal{H}^n$-a.e.\@ on $C(Z')$.
Since $w$ itself is locally Lipschitz continuous (Theorem \ref{thm:RCD_FBP_fund}), it follows that $w$ is invariant in the splitting induced by $g$.
This concludes the splitting for both $C(Z')$ and $w$.
\par
\textbf{Step 4.} 
In this final step, we perform induction on the ordinal $k$ of the stratum.
If $k=0$, then we have a contradiction: $P\in \mathcal{S}^0_{\bar{\epsilon}}(v)\subset \mathcal{S}^0(v)$, but $(C(Z'),d_{C(Z')},w, O')\in \Tan(C(Z),v,P)$ splits off a line.
\par
Suppose $k>0$.
We write $C(Z')\cong \mathbb{R}\times C(Z'')$, where $\mathbb{R}$ is the splitting induced by $g$.
Then by Lemma \ref{lem:restrict_min_is_min}, the section $\tilde{w}:=w|_{\{0\}\times C(Z'')}$
is a global minimizer on $C(Z'')$.
Since $w$ is $\mathbb{R}$-invariant, we see that $(t,y)\in \mathcal{S}^k_{\bar{\epsilon}}(w)\subset \mathbb{R}\times C(Z'')$ if and only if $y\in \mathcal{S}^{k-1}_{\bar{\epsilon}}(\tilde{w})$.
In particular, from \eqref{eq:k'_Hausdorff_measure_3rd_blowup} we get
    \begin{equation*}
        \mathcal{H}^{k'-1}(\mathcal{S}^{k-1}_{\bar{\epsilon}}(\tilde{w})\cap B^{C(Z'')}_1(O''))>0,
    \end{equation*}
where $O''$ is the tip of $C(Z'')$. 
To summarize, we prove that if $u$ is a minimizer of the one-phase problem in a noncollapsed $\mathsf{RCD}(K,n)$ space such that for some $k\in \mathbb{N}_0$ it holds $\dim_H \mathcal{S}^k(u)>k$, then there exists a minimizer $u'$ on a noncollapsed $\mathsf{RCD}(K,n-1)$ space such that $\dim_H \mathcal{S}^{k-1}(u')>k-1$.
This dimension reduction can be iterated a finite number of
times until we reach the case $k = 0$, which we have already discussed above.
\end{proof}

\bibliographystyle{alpha}
\bibliography{ref.bib}

\end{document}